\documentclass[11pt]{amsart}

\usepackage{amsmath}
\usepackage{amssymb}
\usepackage{amscd}
\usepackage{color}
\usepackage{hyperref}
\usepackage{mathrsfs}
\usepackage{eucal}
\usepackage{upgreek}
\usepackage[makeroom]{cancel}
\usepackage[normalem]{ulem}
\usepackage{array}
\usepackage{verbatim}
\usepackage{mathtools}
\usepackage{stmaryrd}

\usepackage{enumitem}

\usepackage{xy}
\xyoption{all}
\usepackage{tikz-cd}

\newcommand{\nc}{\newcommand}

\nc{\md}{\operatorname{-}}

\nc{\CC}{{\mathbb{C}}}
\nc{\DD}{{\mathbb{D}}}
\nc{\LL}{{\mathbb{L}}}
\nc{\RR}{{\mathbb{R}}}
\renewcommand{\P}{{\mathbb{P}}}
\nc{\OO}{{\mathbb{O}}}

\nc{\QQ}{{\mathbb{Q}}}
\nc{\ZZ}{{\mathbb{Z}}}
\nc{\Z}{{\mathbb{Z}}}

\nc{\cA}{{\mathcal{A}}}
\nc{\cB}{{\mathcal{B}}}
\nc{\cC}{{\mathcal{C}}}
\nc{\cD}{{\mathcal{D}}}
\nc{\cE}{{\mathcal{E}}}
\nc{\cF}{{\mathcal{F}}}
\nc{\cG}{{\mathcal{G}}}
\nc{\cH}{{\mathcal{H}}}
\nc{\cI}{{\mathcal{I}}}
\nc{\cJ}{{\mathcal{J}}}
\nc{\cK}{{\mathcal{K}}}
\nc{\cL}{{\mathcal{L}}}
\nc{\cM}{{\mathcal{M}}}
\nc{\cN}{{\mathcal{N}}}
\nc{\cO}{{\mathcal{O}}}
\nc{\cP}{{\mathcal{P}}}
\nc{\cQ}{{\mathcal{Q}}}
\nc{\cR}{{\mathcal{R}}}
\nc{\cS}{{\mathcal{S}}}
\nc{\cT}{{\mathcal{T}}}
\nc{\cU}{{\mathcal{U}}}
\nc{\cV}{{\mathcal{V}}}
\nc{\cW}{{\mathcal{W}}}
\nc{\cX}{{\mathcal{X}}}
\nc{\cY}{{\mathcal{Y}}}
\nc{\cZ}{{\mathcal{Z}}}

\nc{\rc}{{\mathrm{c}}}
\nc{\rd}{{\mathrm{d}}}
\nc{\rf}{{\mathrm{f}}}
\nc{\rh}{{\mathrm{h}}}
\nc{\rrm}{{\mathrm{m}}}
\nc{\rs}{{\mathrm{s}}}
\nc{\rch}{{\mathrm{ch}}}
\nc{\rtd}{{\mathrm{td}}}

\nc{\rA}{{\mathrm{A}}}
\nc{\rB}{{\mathrm{B}}}
\nc{\rC}{{\mathrm{C}}}
\nc{\rD}{{\mathrm{D}}}
\nc{\rE}{{\mathrm{E}}}
\nc{\rF}{{\mathrm{F}}}
\nc{\rG}{{\mathrm{G}}}
\nc{\rH}{{\mathrm{H}}}
\nc{\rI}{{\mathrm{I}}}
\nc{\rJ}{{\mathrm{J}}}
\nc{\rK}{{\mathrm{K}}}
\nc{\rL}{{\mathrm{L}}}
\nc{\rM}{{\mathrm{M}}}
\nc{\rN}{{\mathrm{N}}}
\nc{\rO}{{\mathrm{O}}}
\nc{\rP}{{\mathrm{P}}}
\nc{\rQ}{{\mathrm{Q}}}
\nc{\rR}{{\mathrm{R}}}
\nc{\rS}{{\mathrm{S}}}
\nc{\rT}{{\mathrm{T}}}
\nc{\rU}{{\mathrm{U}}}
\nc{\rV}{{\mathrm{V}}}
\nc{\rW}{{\mathrm{W}}}
\nc{\rX}{{\mathrm{X}}}
\nc{\rY}{{\mathrm{Y}}}
\nc{\rZ}{{\mathrm{Z}}}

\nc{\bA}{{\mathbf{A}}}
\nc{\bB}{{\mathbf{B}}}
\nc{\bC}{{\mathbf{C}}}
\nc{\bD}{{\mathbf{D}}}
\nc{\bE}{{\mathbf{E}}}
\nc{\bF}{{\mathbf{F}}}
\nc{\bG}{{\mathbf{G}}}
\nc{\bH}{{\mathbf{H}}}
\nc{\bI}{{\mathbf{I}}}
\nc{\bJ}{{\mathbf{J}}}
\nc{\bK}{{\mathbf{K}}}
\nc{\bL}{{\mathbf{L}}}
\nc{\bM}{{\mathbf{M}}}
\nc{\bN}{{\mathbf{N}}}
\nc{\bO}{{\mathbf{O}}}
\nc{\bP}{{\mathbf{P}}}
\nc{\bQ}{{\mathbf{Q}}}
\nc{\bR}{{\mathbf{R}}}
\nc{\bS}{{\mathbf{S}}}
\nc{\bT}{{\mathbf{T}}}
\nc{\bU}{{\mathbf{U}}}
\nc{\bV}{{\mathbf{V}}}
\nc{\bW}{{\mathbf{W}}}
\nc{\bX}{{\mathbf{X}}}
\nc{\bY}{{\mathbf{Y}}}
\nc{\bZ}{{\mathbf{Z}}}

\nc{\ba}{{\mathbf{a}}}
\nc{\bb}{{\mathbf{b}}}
\nc{\bc}{{\mathbf{c}}}
\nc{\bd}{{\mathbf{d}}}
\nc{\be}{{\mathbf{e}}}
\nc{\bg}{{\mathbf{g}}}
\nc{\bh}{{\mathbf{h}}}
\nc{\bi}{{\mathbf{i}}}
\nc{\bj}{{\mathbf{j}}}
\nc{\bk}{{\mathbf{k}}}
\nc{\bl}{{\mathbf{l}}}
\nc{\bm}{{\mathbf{m}}}
\nc{\bn}{{\mathbf{n}}}
\nc{\bo}{{\mathbf{o}}}
\nc{\bp}{{\mathbf{p}}}
\nc{\bq}{{\mathbf{q}}}
\nc{\br}{{\mathbf{r}}}
\nc{\bs}{{\mathbf{s}}}
\nc{\bt}{{\mathbf{t}}}
\nc{\bu}{{\mathbf{u}}}
\nc{\bv}{{\mathbf{v}}}
\nc{\bw}{{\mathbf{w}}}
\nc{\bx}{{\mathbf{x}}}
\nc{\by}{{\mathbf{y}}}
\nc{\bz}{{\mathbf{z}}}

\nc{\fA}{{\mathfrak{A}}}
\nc{\fB}{{\mathfrak{B}}}
\nc{\fC}{{\mathfrak{C}}}
\nc{\fD}{{\mathfrak{D}}}
\nc{\fE}{{\mathfrak{E}}}
\nc{\fF}{{\mathfrak{F}}}
\nc{\fG}{{\mathfrak{G}}}
\nc{\fH}{{\mathfrak{H}}}
\nc{\fI}{{\mathfrak{I}}}
\nc{\fJ}{{\mathfrak{J}}}
\nc{\fK}{{\mathfrak{K}}}
\nc{\fL}{{\mathfrak{L}}}
\nc{\fM}{{\mathfrak{M}}}
\nc{\fN}{{\mathfrak{N}}}
\nc{\fO}{{\mathfrak{O}}}
\nc{\fP}{{\mathfrak{P}}}
\nc{\fQ}{{\mathfrak{Q}}}
\nc{\fR}{{\mathfrak{R}}}
\nc{\fS}{{\mathfrak{S}}}
\nc{\fT}{{\mathfrak{T}}}
\nc{\fU}{{\mathfrak{U}}}
\nc{\fV}{{\mathfrak{V}}}
\nc{\fW}{{\mathfrak{W}}}
\nc{\fX}{{\mathfrak{X}}}
\nc{\fY}{{\mathfrak{Y}}}
\nc{\fZ}{{\mathfrak{Z}}}

\nc{\fa}{{\mathfrak{a}}}
\nc{\fb}{{\mathfrak{b}}}
\nc{\fc}{{\mathfrak{c}}}
\nc{\fd}{{\mathfrak{d}}}
\nc{\fe}{{\mathfrak{e}}}
\nc{\ff}{{\mathfrak{f}}}
\nc{\fg}{{\mathfrak{g}}}
\nc{\fh}{{\mathfrak{h}}}
\nc{\fj}{{\mathfrak{j}}}
\nc{\fk}{{\mathfrak{k}}}
\nc{\fl}{{\mathfrak{l}}}
\nc{\fm}{{\mathfrak{m}}}
\nc{\fn}{{\mathfrak{n}}}
\nc{\fo}{{\mathfrak{o}}}
\nc{\fp}{{\mathfrak{p}}}
\nc{\fq}{{\mathfrak{q}}}
\nc{\fr}{{\mathfrak{r}}}
\nc{\fs}{{\mathfrak{s}}}
\nc{\ft}{{\mathfrak{t}}}
\nc{\fu}{{\mathfrak{u}}}
\nc{\fv}{{\mathfrak{v}}}
\nc{\fw}{{\mathfrak{w}}}
\nc{\fx}{{\mathfrak{x}}}
\nc{\fy}{{\mathfrak{y}}}
\nc{\fz}{{\mathfrak{z}}}

\nc{\sA}{{\mathsf{A}}}
\nc{\sB}{{\mathsf{B}}}
\nc{\sC}{{\mathsf{C}}}
\nc{\sD}{{\mathsf{D}}}
\nc{\sE}{{\mathsf{E}}}
\nc{\sF}{{\mathsf{F}}}
\nc{\sG}{{\mathsf{G}}}
\nc{\sH}{{\mathsf{H}}}
\nc{\sI}{{\mathsf{I}}}
\nc{\sJ}{{\mathsf{J}}}
\nc{\sK}{{\mathsf{K}}}
\nc{\sL}{{\mathsf{L}}}
\nc{\sM}{{\mathsf{M}}}
\nc{\sN}{{\mathsf{N}}}
\nc{\sO}{{\mathsf{O}}}
\nc{\sP}{{\mathsf{P}}}
\nc{\sQ}{{\mathsf{Q}}}
\nc{\sR}{{\mathsf{R}}}
\nc{\sS}{{\mathsf{S}}}
\nc{\sT}{{\mathsf{T}}}
\nc{\sU}{{\mathsf{U}}}
\nc{\sV}{{\mathsf{V}}}
\nc{\sW}{{\mathsf{W}}}
\nc{\sX}{{\mathsf{X}}}
\nc{\sY}{{\mathsf{Y}}}
\nc{\sZ}{{\mathsf{Z}}}

\nc{\sa}{{\mathsf{a}}}
\nc{\sd}{{\mathsf{d}}}
\nc{\se}{{\mathsf{e}}}
\nc{\sg}{{\mathsf{g}}}
\nc{\sh}{{\mathsf{h}}}
\nc{\si}{{\mathsf{i}}}
\nc{\sj}{{\mathsf{j}}}
\nc{\sk}{{\mathsf{k}}}
\nc{\sm}{{\mathsf{m}}}
\nc{\sn}{{\mathsf{n}}}
\nc{\so}{{\mathsf{o}}}
\nc{\sq}{{\mathsf{q}}}
\nc{\sr}{{\mathsf{r}}}
\nc{\st}{{\mathsf{t}}}
\nc{\su}{{\mathsf{u}}}
\nc{\sv}{{\mathsf{v}}}
\nc{\sw}{{\mathsf{w}}}
\nc{\sx}{{\mathsf{x}}}
\nc{\sy}{{\mathsf{y}}}
\nc{\sz}{{\mathsf{z}}}

\nc{\oA}{{\overline{A}}}
\nc{\oB}{{\overline{B}}}
\nc{\oC}{{\overline{C}}}
\nc{\oD}{{\overline{D}}}
\nc{\oE}{{\overline{E}}}
\nc{\oF}{{\overline{F}}}
\nc{\oG}{{\overline{G}}}
\nc{\oH}{{\overline{H}}}
\nc{\oI}{{\overline{I}}}
\nc{\oJ}{{\overline{J}}}
\nc{\oK}{{\overline{K}}}
\nc{\oL}{{\overline{L}}}
\nc{\oM}{{\overline{M}}}
\nc{\oN}{{\overline{N}}}
\nc{\oO}{{\overline{O}}}
\nc{\oP}{{\overline{P}}}
\nc{\oQ}{{\overline{Q}}}
\nc{\oR}{{\overline{R}}}
\nc{\oS}{{\overline{S}}}
\nc{\oT}{{\overline{T}}}
\nc{\oU}{{\overline{U}}}
\nc{\oV}{{\overline{V}}}
\nc{\oW}{{\overline{W}}}
\nc{\oX}{{\overline{X}}}
\nc{\oY}{{\overline{Y}}}
\nc{\oZ}{{\overline{Z}}}

\nc{\oa}{{\overline{a}}}
\nc{\ob}{{\overline{b}}}
\nc{\oc}{{\overline{c}}}
\nc{\od}{{\overline{d}}}
\nc{\of}{{\overline{f}}}
\nc{\og}{{\overline{g}}}
\nc{\oh}{{\overline{h}}}
\nc{\oi}{{\overline{i}}}
\nc{\oj}{{\overline{j}}}
\nc{\ok}{{\overline{k}}}
\nc{\ol}{{\overline{l}}}
\nc{\om}{{\overline{m}}}
\nc{\on}{{\overline{n}}}
\nc{\oo}{{\overline{o}}}
\nc{\op}{{\overline{p}}}
\nc{\oq}{{\overline{q}}}
\nc{\os}{{\overline{s}}}
\nc{\ot}{{\overline{t}}}
\nc{\ou}{{\overline{u}}}
\nc{\ov}{{\overline{v}}}
\nc{\ow}{{\overline{w}}}
\nc{\ox}{{\overline{x}}}
\nc{\oy}{{\overline{y}}}
\nc{\oz}{{\overline{z}}}

\nc{\tA}{{\tilde{A}}}
\nc{\tB}{{\tilde{B}}}
\nc{\tC}{{\tilde{C}}}
\nc{\tD}{{\tilde{D}}}
\nc{\tE}{{\tilde{E}}}
\nc{\tF}{{\tilde{F}}}
\nc{\tG}{{\tilde{G}}}
\nc{\tH}{{\tilde{H}}}
\nc{\tI}{{\tilde{I}}}
\nc{\tJ}{{\tilde{J}}}
\nc{\tK}{{\tilde{K}}}
\nc{\tL}{{\tilde{L}}}
\nc{\tM}{{\tilde{M}}}
\nc{\tN}{{\tilde{N}}}
\nc{\tO}{{\tilde{O}}}
\nc{\tP}{{\tilde{P}}}
\nc{\tQ}{{\tilde{Q}}}
\nc{\tR}{{\tilde{R}}}
\nc{\tS}{{\tilde{S}}}
\nc{\tT}{{\tilde{T}}}
\nc{\tU}{{\tilde{U}}}
\nc{\tV}{{\tilde{V}}}
\nc{\tW}{{\tilde{W}}}
\nc{\tX}{{\tilde{X}}}
\nc{\tY}{{\tilde{Y}}}
\nc{\tZ}{{\tilde{Z}}}

\nc{\tfD}{{\tilde{\fD}}}
\nc{\tcA}{{\tilde{\cA}}}
\nc{\tcB}{{\tilde{\cB}}}
\nc{\tcC}{{\tilde{\cC}}}
\nc{\tcD}{{\tilde{\cD}}}
\nc{\tcE}{{\tilde{\cE}}}
\nc{\tcF}{{\tilde{\cF}}}
\nc{\tcM}{{\tilde{\cM}}}
\nc{\tcP}{{\tilde{\cP}}}
\nc{\tcT}{{\tilde{\cT}}}

\nc{\ta}{{\tilde{a}}}
\nc{\tb}{{\tilde{b}}}
\nc{\tc}{{\tilde{c}}}
\nc{\td}{{\tilde{d}}}
\nc{\te}{{\tilde{e}}}
\nc{\tf}{{\tilde{f}}}
\nc{\tg}{{\tilde{g}}}
\nc{\ti}{{\tilde{\imath}}}
\nc{\tj}{{\tilde{j}}}
\nc{\tk}{{\tilde{k}}}
\nc{\tl}{{\tilde{l}}}
\nc{\tm}{{\tilde{m}}}
\nc{\tn}{{\tilde{n}}}
\nc{\tp}{{\tilde{p}}}
\nc{\tq}{{\tilde{q}}}
\nc{\tr}{{\tilde{r}}}
\nc{\ts}{{\tilde{s}}}
\nc{\tu}{{\tilde{u}}}
\nc{\tv}{{\tilde{v}}}
\nc{\tw}{{\tilde{w}}}
\nc{\tx}{{\tilde{x}}}
\nc{\ty}{{\tilde{y}}}
\nc{\tz}{{\tilde{z}}}

\nc{\hA}{{\hat{A}}}
\nc{\hB}{{\hat{B}}}
\nc{\hC}{{\hat{C}}}
\nc{\hD}{{\hat{D}}}
\nc{\hE}{{\hat{E}}}
\nc{\hF}{{\hat{F}}}
\nc{\hG}{{\hat{G}}}
\nc{\hH}{{\hat{H}}}
\nc{\hI}{{\hat{I}}}
\nc{\hJ}{{\hat{J}}}
\nc{\hK}{{\hat{K}}}
\nc{\hL}{{\hat{L}}}
\nc{\hM}{{\hat{M}}}
\nc{\hN}{{\hat{N}}}
\nc{\hO}{{\hat{O}}}
\nc{\hP}{{\hat{P}}}
\nc{\hQ}{{\hat{Q}}}
\nc{\hR}{{\hat{R}}}
\nc{\hS}{{\hat{S}}}
\nc{\hT}{{\hat{T}}}
\nc{\hU}{{\hat{U}}}
\nc{\hV}{{\hat{V}}}
\nc{\hW}{{\hat{W}}}
\nc{\hX}{{\widehat{X}}}
\nc{\hY}{{\hat{Y}}}
\nc{\hZ}{{\hat{Z}}}

\nc{\ha}{{\hat{a}}}
\nc{\hb}{{\hat{b}}}
\nc{\hc}{{\hat{c}}}
\nc{\hd}{{\hat{d}}}
\nc{\he}{{\hat{e}}}
\nc{\hg}{{\hat{g}}}
\nc{\hh}{{\hat{h}}}
\nc{\hi}{{\hat{i}}}
\nc{\hj}{{\hat{j}}}
\nc{\hk}{{\hat{k}}}
\nc{\hl}{{\hat{l}}}
\nc{\hm}{{\hat{m}}}
\nc{\hn}{{\hat{n}}}
\nc{\ho}{{\hat{o}}}
\nc{\hp}{{\hat{p}}}
\nc{\hq}{{\hat{q}}}
\nc{\hr}{{\hat{r}}}
\nc{\hs}{{\hat{s}}}
\nc{\hu}{{\hat{u}}}
\nc{\hv}{{\hat{v}}}
\nc{\hw}{{\hat{w}}}
\nc{\hx}{{\hat{x}}}
\nc{\hy}{{\hat{y}}}
\nc{\hz}{{\hat{z}}}

\nc{\hcC}{{\widehat{\cC}}}
\nc{\hcT}{{\widehat{\cT}}}

\nc{\eps}{\upepsilon}
\nc{\lan}{\big\langle}
\nc{\ran}{\big\rangle}
\nc{\kk}{{\Bbbk}}
\nc{\io}{\upiota}
\nc{\Kr}{\mathsf{Kr}}
\nc{\cKr}{\mathcal{K}\!\mathit{r}}

\nc{\Dm}{\bD^{-}}
\nc{\Db}{\bD^{\mathrm{b}}}
\nc{\Dbc}{\bD^{\mathrm{b}}_{\mathrm{c}}}
\nc{\Dp}{\bD^{\mathrm{perf}}}
\nc{\Dperf}{\bD^{\mathrm{perf}}}
\nc{\Dqc}{\bD_{\mathrm{qc}}}
\nc{\Du}{\bD}
\nc{\Dsing}{\bD^{\mathrm{sg}}}
\nc{\Dg}{\bD^{\mathrm{sg}}}
\DeclareMathOperator{\duk}{\mathbf{d}_\Bbbk}

\def\ol{\overline}

\newcommand{\pf}{{\mathrm{perf}}}
\newcommand{\homfin}{{\mathrm{hf}}}
\newcommand{\hf}{{\mathrm{hf}}}
\newcommand{\lhf}{{\mathrm{lhf}}}
\newcommand{\rhf}{{\mathrm{rhf}}}
\newcommand{\sing}{{\mathrm{sg}}}
\newcommand{\opp}{\mathrm{op}}

\nc{\Rn}{\rR_{\mathrm{node}}}
\nc{\Cn}{\cC_{\mathrm{node}}}
\nc{\Dfd}[1]{\bD_{\mathrm{fd}}(#1)}

\def\bw#1#2{\textstyle{\bigwedge\hskip-0.9mm^{#1}}\hskip0.2mm{#2}}

\nc{\xrightiso}[1]{ \xrightarrow[{\ \raisebox{0.5ex}[0ex][0ex]{$\sim$}\ }]{#1} }

\nc{\thick}{\mathbf{thick}}

\DeclareMathOperator{\ev}{\mathbf{ev}}
\DeclareMathOperator{\coev}{\mathbf{coev}}

\DeclareMathOperator{\Hom}{\mathrm{Hom}}

\DeclareMathOperator{\Ext}{\mathrm{Ext}}

\DeclareMathOperator{\End}{\mathrm{End}}

\DeclareMathOperator{\RHom}{\mathrm{RHom}}
\DeclareMathOperator{\cRHom}{\mathrm{R}\mathcal{H}\mathit{om}}

\DeclareMathOperator{\CM}{\mathrm{CM}}

\DeclareMathOperator{\Ker}{\mathrm{Ker}}

\DeclareMathOperator{\Ima}{\mathrm{Im}}
\DeclareMathOperator{\Cone}{\mathrm{Cone}}

\DeclareMathOperator{\Ind}{\mathrm{Ind}}

\DeclareMathOperator{\Res}{\mathrm{Res}}

\DeclareMathOperator{\id}{\mathrm{id}}
\DeclareMathOperator{\rad}{\mathrm{rad}}

\DeclareMathOperator{\Adm}{\mathrm{Adm}}
\DeclareMathOperator{\LAdm}{\mathrm{LAdm}}
\DeclareMathOperator{\RAdm}{\mathrm{RAdm}}
\DeclareMathOperator{\Sub}{\mathrm{Sub}}

\def\Pinfty#1{\P^{\infty,{#1}}}

\theoremstyle{plain}

\newtheorem{theorem}{Theorem}[section]

\newtheorem{lemma}[theorem]{Lemma}
\newtheorem{proposition}[theorem]{Proposition}
\newtheorem{corollary}[theorem]{Corollary}

\theoremstyle{definition}

\newtheorem{definition}[theorem]{Definition}

\newtheorem{example}[theorem]{Example}

\theoremstyle{remark}

\newtheorem{remark}[theorem]{Remark}

\newenvironment{renumerate}{\begin{enumerate}[label={\textup{(\roman*)}}]}{\end{enumerate}}
\newenvironment{aenumerate}{\begin{enumerate}[label={\textup{(\alph*)}}]}{\end{enumerate}}

\title[Homologically finite-dimensional objects in triangulated categories]%
{Homologically finite-dimensional objects\\[1ex]in triangulated categories}
\author{Alexander Kuznetsov}
\address{{\sloppy
\parbox{0.9\textwidth}{
Algebraic Geometry Section, Steklov Mathematical Institute of Russian Academy of Sciences,\\
8 Gubkin str., Moscow 119991 Russia
\\[5pt]
Laboratory of Algebraic Geometry, National Research University Higher School of Economics, Russian Federation
}\bigskip}}
\email{akuznet@mi-ras.ru}
\date{}
\thanks{A.K. was partially supported by the HSE University Basic Research Program.
	E.S. is supported by the EPSRC grant
    EP/T019379/1 ``Derived categories and algebraic K-theory of singularities'', and by the
    ERC Synergy grant ``Modern Aspects of Geometry: Categories, Cycles and Cohomology of Hyperk\"ahler Varieties".
	}
\author{Evgeny Shinder}
\address{School of Mathematical and Physical Sciences, University of Sheffield,
Hounsfield Road, S3 7RH, UK, and
Hausdorff Center for Mathematics
at the University of Bonn, Endenicher Allee 60, 53115.}
\email{eugene.shinder@gmail.com}

\begin{document}

\begin{abstract}
In this paper we investigate homologically finite-dimensional objects
in the derived category of a given small dg-enhanced triangulated category.
Using these we define reflexivity, hfd-closedness, and the Gorenstein property for triangulated categories,
and discuss crepant categorical contractions.
We illustrate the introduced notions on examples of categories of geometric and algebraic origin
and provide geometric applications.
In particular, we apply our results to prove a bijection between semiorthogonal decompositions 
of the derived category of a singular variety and the derived category of its smoothing with support on the central fiber.
\end{abstract}

\maketitle

\setcounter{tocdepth}{2}
\tableofcontents

\section{Introduction}

Let~$\cT$ be a small \emph{dg-enhanced} triangulated category.
In this paper we study the subcategory~$\Dfd{\cT}$ of homologically finite-dimensional objects 
in the derived category~$\bD(\cT)$ of right dg-modules over~$\cT$ and the operation~$\cT \mapsto \Dfd{\cT}$.

The main definition is very simple.
We say that an object~$M \in \bD(\cT)$ is {\sf homologically finite-dimensional} 
if for any object~$t \in \cT$ we have~$M(t) \in \Db(\kk)$,
i.e., $M(t)$ is bounded and \emph{finite-dimensional} complex of~$\kk$-vector spaces, see Definition~\ref{def:rhfd-lhf}.
We denote by~$\Dfd{\cT} \subset \bD(\cT)$ the subcategory of all homologically finite-dimensional objects over~$\cT$.
Obviously, this is a small dg-enhanced triangulated subcategory.

\subsection{Semiorthogonal decompositions and reflexivity}

The operation~$\cT \mapsto \Dfd{\cT}$ behaves nicely with respect to semiorthogonal decompositions ---
it is easy to show that if~$\cT = \langle \cA_1, \dots, \cA_n \rangle$ is a semiorthogonal decomposition 
(without any extra admissibility assumptions) 
then~$\Dfd{\cT}$ also has 
a semiorthogonal decomposition
\begin{equation*}
\Dfd{\cT} = \langle \Dfd{\cA_n}, \dots, \Dfd{\cA_1} \rangle
\end{equation*}
(note that the order of the components is inverted), see Lemma~\ref{lem:sod-rhfd-lhf}.
In particular, it follows that the operation~$\cA \mapsto \Dfd{\cA}$ defines maps
\begin{equation}
\label{eq:ladm-radm}
\LAdm(\cT) \to \RAdm(\Dfd{\cT})
\qquad\text{and}\qquad 
\RAdm(\cT) \to \LAdm(\Dfd{\cT})
\end{equation}
between the sets of all left or right admissible subcategories of~$\cT$ and~$\Dfd{\cT}$, respectively.
We show that, under an appropriate hypothesis about~$\cT$, these operations are bijective and mutually inverse.

To show this we note that 
an important feature of the 
definition of the category~$\Dfd{\cT}$, already mentioned above, is that it is a small triangulated dg-category,
hence the operation~$\cT \mapsto \Dfd{\cT}$ may be iterated.
In particular, one can iterate the operations~\eqref{eq:ladm-radm} and study the compositions.
This leads us to an important definition: 
we say that a small dg-enhanced triangulated category~$\cT$ is {\sf reflexive} if
\begin{equation*}
\Dfd{\Dfd{\cT}} \simeq \cT 
\end{equation*}
via a natural functor (see  Definition~\ref{def:reflexivity} for details).
We believe that reflexivity is a very interesting and useful notion, and we prove a few nice properties enjoyed by reflexive categories.
For instance, we show that, whenever~$\cT$ is reflexive,
the opposite category~$\cT^\opp$ and the category~$\Dfd{\cT}$ are also reflexive 
(Lemma~\ref{lem:reflexivity-opp} and Lemma~\ref{lem:lhfd-lhfrefl}, respectively).

Furthermore, if~$\cT$ is a reflexive category we prove that the composition of the operations~\eqref{eq:ladm-radm} 
\begin{equation*}
\LAdm(\cT) \to \LAdm(\Dfd{\Dfd{\cT}} 
\qquad 
\cA \mapsto \Dfd{\Dfd{\cA}} \subset \Dfd{\Dfd{\cT}},
\end{equation*}
coincides with the map defined by the equivalence~$\cT \xrightiso{\ } \Dfd{\Dfd{\cT}}$,
hence, indeed, the operations~\eqref{eq:ladm-radm} are bijections, see Theorem~\ref{thm:lhfd-right-bijection}.
It also follows that the reflexivity property is inherited by any left or right admissible subcategory.

One simple corollary of the bijections~$\LAdm(\cT) \cong \RAdm(\Dfd{\cT})$
is that (semiorthogonal) indecomposability of a reflexive category~$\cT$
is equivalent to indecomposability of the category~$\Dfd{\cT}$, 
see Corollaries~\ref{cor:Db-ind}--\ref{cor:indecomposability-curves}
for geometric applications of this observation.

A similar argument allows us to establish a bijection between sets of isomorphism classes of dg-functors~$\cT_1 \to \cT_2$
and~$\Dfd{\cT_2} \to \Dfd{\cT_1}$ for reflexive categories~$\cT_1$ and~$\cT_2$, see Corollary~\ref{cor:equivalence-Refl}.

\subsection{HFD-closed and Gorenstein categories}

If all homologically finite-dimensional dg-modules over~$\cT$ are representable, then the category~$\Dfd{\cT}$ is  contained in~$\cT$.
If this property holds both for~$\cT$ and the opposite category~$\cT^\opp$, we call~$\cT$ {\sf hfd-closed} (Definition~\ref{def:hfd-closed}).
It is easy to see that this property holds for any (homologically) \emph{smooth} idempotent complete dg-category 
(see Lemma~\ref{lemma:sp-hfd}\ref{item:hfd-smooth}).

For hfd-closed categories many definitions and constructions simplify.
For instance, if~$\cT$ is hfd-closed, the abstract categories~$\Dfd{\cT}$ and~$\Dfd{\cT^\opp}$ 
can be replaced by simpler subcategories~\mbox{$\cT^\rhf,\cT^\lhf \subset \cT$}, see~\eqref{eq:ct-rhf-lhf}.
Moreover, if~$\cA \subset \cT$ is a left or right admissible subcategory in an hfd-closed category~$\cT$, 
the operations~\eqref{eq:ladm-radm} take the simpler form
\begin{equation*}
\cA \mapsto \cA \cap \cT^\rhf 
\qquad\text{and}\qquad 
\cA \mapsto \cA \cap \cT^\lhf,
\end{equation*}
respectively, see Proposition~\ref{prop:bijectionsubcat-hfd-closed}\ref{it:reflexive-hfd-closed-bijection}.
Moreover, if~$\cT$ is simultaneously hfd-closed and reflexive, 
$\cT = \langle \cA, \cB \rangle$ is a semiorthogonal decomposition, 
and one of its components is admissible,
then it follows that 
\begin{equation*}
\cT^\lhf = \langle \cA \cap \cT^\lhf, \cB \cap \cT^\lhf \rangle
\qquad\text{or}\qquad 
\cT^\rhf = \langle \cA \cap \cT^\rhf, \cB \cap \cT^\rhf \rangle,
\end{equation*}
if~$\cA$ is admissible or~$\cB$ is admissible, respectively, see Corollary~\ref{cor:reflexive-hfd-closed-ct-hf}.
Note that in this case the order of the components is not inverted.

If a category~$\cT$ is hfd-closed and there is an equality~$\cT^\rhf = \cT^\lhf$ of subcategories in~$\cT$,
we call~$\cT$ {\sf Gorenstein} (Definition~\ref{def:gorenstein}).
We show that for any Gorenstein category~$\cT$ 
the category
\begin{equation*}
\cT^\hf \coloneqq \cT^\rhf = \cT^\lhf \subset \cT
\end{equation*}
has a Serre functor~$\bS_{\cT^\hf}$, 
and if~$\cT$ is reflexive, it enjoys a stronger Serre duality property --- 
there is an autoequivalence~$\bS_\cT$ of~$\cT$ and a functorial (in both arguments) isomorphism
\begin{equation*}
\Hom_\cT(t_1,t_2)^\vee \cong \Hom_\cT(t_2,\bS_\cT(t_1)),
\end{equation*}
whenever either of the objects~$t_1,t_2$ belongs to the subcategory~$\cT^\hf \subset \cT$;
moreover, $\bS_\cT$ preserves~$\cT^\hf$ and the restriction~$\bS_\cT\vert_{\cT^\hf}$ 
is isomorphic to the Serre functor~$\bS_{\cT^\hf}$ of~$\cT^\hf$, see Proposition~\ref{prop:serre-functor}.
Another nice feature of Gorenstein categories is that the two operations
relating left or right admissible subcategories in~$\cT$ and~$\cT^\hf$ agree,
and therefore they preserve admissibility.

As an upshot of this discussion, we suggest to think of an hfd-closed reflexive category~$\cT$
as the 
bounded derived category of coherent sheaves on a
\emph{proper noncommutative variety};
then the category~$\cT^\hf$ plays the role of the category of perfect complexes on the same variety.
We support this point of view by showing that the bounded derived category~$\Db(X)$ for a proper variety~$X$ 
is reflexive and hfd-closed and~$\Db(X)^\hf \simeq \Dp(X)$,
see~\S\ref{ss:intro-geometry} below for this and other examples and~\S\ref{sec:examples} for more detail.

\subsection{Categorical contractions and crepancy}

In~\S\ref{sec:ccc} we apply the machinery of homologically finite-dimensional objects 
in the situation of a dg-enhanced triangulated functor~$\pi_* \colon \tcT \to \cT$ between triangulated dg-categories,
thinking of it as the pushforward functor~$\pi_* \colon \Db(\tX) \to \Db(X)$
for a morphism~\mbox{$\pi \colon \tX \to X$} of algebraic varieties.

More precisely, in~\S\ref{ss:cc} we use our results about
homologically finite-dimensional objects and the notion of Gorenstein category
to prove some nice properties of \emph{categorical contractions} defined in~\cite[Definition~1.10]{KSabs}
(see also Definition~\ref{def:cc}).
In particular, we show that a categorical contraction from an hfd-closed category
automatically has fully faithful adjoint functors on subcategories of homologically finite-dimensional objects 
(Proposition~\ref{prop:hfd-images}).

Furthermore, in~\S\ref{ss:ccc} we define \emph{crepancy} of a categorical contraction (Definition~\ref{def:crepancy})
and show that it is equivalent to a simple condition on the kernel subcategory~$\Ker(\pi_*) \subset \tcT$ 
(see Lemma~\ref{lem:crepancy-criterion} and Lemma~\ref{lemma:Serre-kernel}).
In particular, if~$\Ker(\pi_*)$ is generated by spherical objects, then~$\pi_*$ is crepant; 
see~\cite[\S5]{KSabs} or~\cite{CGLMMPS} for examples of such categorical resolutions.
Moreover, in Corollary~\ref{cor:crepant-contractions-vs-resolutions} we
relate crepant categorical contractions to weakly crepant categorical resolutions from~\cite{K08}.

\subsection{Geometric and algebraic examples}
\label{ss:intro-geometry}

In~\S\ref{sec:geometry} we illustrate the notions and results explained above 
in the case of the derived category of a projective scheme~$X$ over a perfect field.
More precisely, we show in Proposition~\ref{prop:hfd-geometric} 
that the category of perfect complexes~$\Dp(X)$ on~$X$ is proper and 
\begin{equation}
\label{eq:dpx-dbx}
\Dfd{\Dp(X)} \simeq \Db(X),
\end{equation}
while the bounded derived category~$\Db(X)$ of coherent sheaves is hfd-closed and 
\begin{equation}
\label{eq:dbx-dpx}
\Db(X)^\lhf = \Dp(X),
\qquad 
\Db(X)^\rhf = \Dp(X) \otimes \omega^\bullet_X,
\end{equation}
where~$\omega^\bullet_X$ is the dualizing complex.
In particular, both~$\Dp(X)$ and~$\Db(X)$ are reflexive, and~$\Db(X)$ is a Gorenstein category if and only if~$X$ is a Gorenstein scheme.
As a consequence of these observations, we deduce in Corollary~\ref{ex:DbDperf-decomp} bijections
\begin{equation}
\label{eq:radm-ladm}
\RAdm(\Db(X)) \cong \LAdm(\Dp(X))
\qquad\text{and}\qquad
\LAdm(\Db(X)) \cong \RAdm(\Dp(X) \otimes \omega^\bullet_X)
\end{equation}
between the sets of right admissible and left admissible subcategories,
and, if the scheme~$X$ is Gorenstein, a bijection~$\Adm(\Db(X)) \cong \Adm(\Dp(X))$ between the sets of admissible subcategories.
Moreover, in Corollaries~\ref{cor:Db-ind}, \ref{cor:cm-db-indecomposable}, and~\ref{cor:indecomposability-curves}
we use this to prove indecomposability of~$\Db(X)$ when~$\Dp(X)$ is known to be indecomposable;
e.g., for Cohen--Macaulay varieties with small base locus of the dualizing sheaf and for nodal curves.

In~\S\ref{sec:algebra} we illustrate our results for categories of algebraic nature,
namely for the derived category of proper connective dg-algebras~$A$:
we show in Proposition~\ref{prop:hfd-algebraic} that the category of perfect $A$-modules~$\Dp(A)$ is proper and 
\begin{equation*}
\Dfd{\Dp(A)} \simeq \Db(A),
\end{equation*}
while the category~$\Db(A)$ of dg-modules with finite-dimensional total cohomology
is hfd-closed and 
\begin{equation*}
\Db(A)^\lhf = \thick(A) = \Dp(A),
\qquad 
\Db(A)^\rhf = \thick(A^\vee),
\end{equation*}
where~$\thick(-)$ stands for the thick envelope.
In particular, both~$\Dp(A)$ and~$\Db(A)$ are reflexive, and~$\Db(A)$ is a Gorenstein category 
if and only if~$A$ is a Gorenstein dg-algebra in the sense of~\cite{Jin}.

While we only studied two sorts of examples coming from algebra and geometry,
we expect our techniques to be applicable in a wider generality.
In particular, it is very interesting to interpret reflexivity, hfd-closedness, and the Gorenstein property
for the (wrapped) Fukaya category of a (noncompact) symplectic variety.
See~\cite[\S4.4]{LP18} and in particular~\cite[Proposition~4.4.1]{LP18} 
for a related duality result about the partially wrapped and infinitesimally wrapped Fukaya categories of a punctured curve.
Another interesting question is to study the category~$\Dfd{\cT}$ 
where~$\cT$ is the Voevodsky category of geometric mixed motives.

\subsection{An extension result}

We conclude the paper with a geometric application of our results
to the deformation theory of semiorthogonal decompositions 
of a special singular fiber of a morphism.
We formulate the following theorem in a slightly more general setup.

\begin{theorem}
\label{thm:bijection-subcat-deform}
Let~$\io \colon X \hookrightarrow \cX$ be an embedding of a projective Gorenstein scheme~$X$ over a perfect field
into a smooth quasiprojective variety~$\cX$ such that~$X \subset \cX$ is a Cartier divisor linearly equivalent to zero.
Then there is a commutative diagram of bijective maps
\begin{equation}
\label{eq:upsilon-diagram}
\vcenter{\xymatrix{
\Adm(\Db(X))  \ar[rd]_{\Upsilon_{\io_*}}
\ar[rr]^{\Upsilon^\pf}
&&
\Adm(\Dp(X)) 
\\
&
\Adm(\Db_X(\cX))
\ar[ur]_{\Upsilon_{\io^*}}
}}
\end{equation}
preserving semiorthogonal decompositions with two components, one of which is admissible, where
\begin{equation*}
\Upsilon_{\io_*}(\cA) \coloneqq \thick(\io_*(\cA)),
\qquad
\Upsilon_{\io^*}(\cA) \coloneqq \thick(\io^*(\cA)),
\qquad\text{and}\qquad 
\Upsilon^\pf(\cA) \coloneqq \cA \cap \Dp(X),
\end{equation*}
and~$\Db_X(\cX)$ is the full subcategory of~$\Db(\cX)$ of objects set-theoretically supported on~$X$.
\end{theorem}

This theorem generalizes our observation from~\cite{KSabs} 
saying that in the above situation
so-called~$\Pinfty{2}$-objects on~$X$ 
correspond to exceptional objects on~$\cX$ scheme-theoretically supported on~$X \subset \cX$.
We explain this in Corollary~\ref{cor:bijection-pinfty}.

\subsection{Relation to other work}

Our subcategory~$\cT^\lhf \subset \cT$ 
is closely related to the subcategory of \emph{homologically finite} objects in~$\cT$ defined by Orlov in~\cite{Orl06},
and in the case where all~$\Hom$-spaces in~$\cT$ are finite-dimensional, these two subcategories coincide.
Orlov proved that a semiorthogonal decomposition of $\cT$ with admissible components
induces a semiorthogonal decomposition of homologically finite objects.
We generalize this to right and left homologically finite objects of an hfd-closed category in Corollary~\ref{cor:reflexive-hfd-closed-ct-hf}.

The notion of homologically finite objects was also used by Lunts in~\cite[\S6.3]{Lunts},
where it was observed that the Gorenstein property of a scheme~$X$ can be interpreted as the equality 
of the subcategories of homologically finite objects in~$\Db(X)$ and its dual version.
Partially defined Serre functors for triangulated categories were studied earlier by Chen~\cite{Chen-Serre} in the algebraic context
and by Ballard~\cite{Ballard} in the geometric context.
Our concept of a Gorenstein category provides a setting where the Serre functor behaves in the same way as it does on a
Gorenstein projective variety or for a finite dimensional Gorenstein algebra.

The idea of thinking of~$\Db(X)$ and~$\Dp(X)$ as ``mutually dual'' categories is, of course, not new;
it goes back to the works of Bondal--Van den Bergh~\cite{Bondal-vdB}, Orlov~\cite{Orl06}, 
Rouquier~\cite{Rouquier1, Rouquier2}, Ballard~\cite{Ballard} and~Ben-Zvi--Nadler--Preygel~\cite{BZNP} in the largest generality.
In particular, the term ``reflexive category'' appears in~\cite[Remark~1.2.6]{BZNP}.

One of the ways to express the duality between~$\Db(X)$ and~$\Dperf(X)$, which generalizes~\cite{Bondal-vdB}
is Neeman's theory of \emph{approximable triangulated categories}, see~\cite{Neeman2, Neeman3, Neeman4}.
In particular, under quite general assumptions, Neeman showed in~\cite[Application~1.4(iii),(iv)]{Neeman3} that~$\Db(X)$ (resp.~$\Dperf(X)$)
can be identified with the category of cohomological (resp.~homological) finite functors on~$\Dperf(X)$ (resp.~$\Db(X)$).
The advantage of our approach is that the category~$\Dfd{\cT}$ of homologically finite-dimensional objects 
is always a dg-enhanced triangulated category 
(in contrast to the category of homological functors on a triangulated category 
which is not known to possess a natural triangulated structure);
this allows us to define and study the reflexivity property 
and the action of the operation~\mbox{$\cT \mapsto \Dfd{\cT}$} on semiorthogonal decompositions.

Correspondence between admissible subcategories in $\Db(X)$ and $\Dperf(X)$
have been studied in \cite{Orl06}, \cite{KKS20}, \cite{B22}.
In particular, the
first bijection in~\eqref{eq:radm-ladm} 
was obtained independently by Bondarko~\cite{B22} in a more general situation.
Our results provide a natural 
categorical perspective on this bijection,
which is in addition symmetric (meaning that it works for both right and left admissible categories). 

\subsubsection*{Notation and conventions}

Throughout the paper~$\Bbbk$ denotes a base field.

For any category~$\cT$ we denote by~$\cT^\opp$ the opposite category.
For a subset $S \subset \cT$ in a triangulated category we denote by~$\thick(S) \subset \cT$ 
the thick subcategory generated by~$S$, 
i.e., the smallest closed under direct summands triangulated subcategory of~$\cT$ containing~$S$.

Admissible subcategories of triangulated categories are assumed to be \emph{strict}, that is closed under isomorphism.
We write~$\LAdm(\cT)$, $\RAdm(\cT)$ and~$\Adm(\cT) = \LAdm(\cT) \cap \RAdm(\cT)$ for the sets 
of all left admissible, right admissible, and admissible subcategories of a triangulated category~$\cT$.
For a subcategory~$\cA \subset \cT$ we denote by~$\cA^\perp$ and~${}^\perp\cA$ the right and left orthogonals of~$\cA$ in~$\cT$.
We write~$\cT = \langle \cA_1, \dots, \cA_m \rangle$ for a semiorthogonal decomposition with components~$\cA_1,\dots,\cA_m$.

We say that a diagram of functors is commutative when it is commutative up to isomorphism.

\subsubsection*{Acknowledgements}

We would like to thank 
Alexey Bondal,
Sasha Efimov, 
Haibo Jin,
Martin Kalck,
Bernhard Keller,
Ana Cristina L\'{o}pez Mart\'{i}n, 
Shinnosuke Okawa,
Dima Orlov, 
Amnon Neeman,
Nebojsa Pavic,
Alex Perry, 
Greg Stevenson, and
Michael Wemyss
for their help and interest in this work.


\section{Preliminaries}
\label{sec:prelim}

In this section we recall some facts about dg-categories and their derived categories;
we refer to~\cite{Kel06} and~\cite[\S3]{KL} for a more detailed treatment of the subject.

Recall that a $\kk$-linear category~$\cD$ is {\sf cocomplete} if it admits arbitrary direct sums.
Further, a $\kk$-linear functor~$\cD_1 \to \cD_2$ between cocomplete categories is {\sf continuous} 
if it commutes with arbitrary direct sums.
Furthermore, an object~$M \in \cD$ of a cocomplete triangulated category is {\sf compact}
if the functor~\mbox{$\Hom_\cD(M,-) \colon \cD \to \kk\text{-}\mathrm{mod}$} is continuous,
and a {\sf set of compact generators} in a cocomplete triangulated category~$\cD$
is a set of compact objects~$S \subset \cD$ such that the orthogonal category vanishes:
\begin{equation*}
S^\perp \coloneqq \{ M \in \cD \mid \Hom_\cD(s[i],M) = 0 \text{\ for all~$s \in S$ and~$i \in \ZZ$}\}  = 0.
\end{equation*}

From now on let~$\cT$ be an essentially small dg-enhanced $\kk$-linear triangulated category.
We denote by~$\bD(\cT)$ the derived category of all \emph{right} dg-modules over~$\cT$,
i.e., the homotopy category of all dg-functors~$\cT^\opp \to \bD(\kk)$ localized with respect to acyclic dg-modules;
the category~$\bD(\cT)$ is cocomplete.
We denote by~$\RHom_\cT(-,-)$ the complex of morphisms between two objects of~$\cT$.
Note that its cohomology in degree zero coincides with the $\Hom$-space in (the homotopy category of)~$\cT$:
\begin{equation*}
\Hom_\cT(-, -) \coloneqq \rH^0(\RHom_\cT(-, -)).
\end{equation*}

Recall that a dg-module over~$\cT$ 
quasiisomorphic to the dg-module
\begin{equation}
\label{eq:def-yoneda}
\bh_{\cT,t}(-) \coloneqq \RHom_\cT(-,t),
\qquad 
t \in \cT
\end{equation}
is called {\sf representable}.
If the category~$\cT$ is clear from context we will sometimes abbreviate this notation to simply~$\bh_t$.
The functor
\begin{equation*}
\bh_\cT \colon {}
\cT \to \bD(\cT),
\qquad 
t \mapsto \bh_{\cT,t}
\end{equation*}
from~$\cT$ to the derived category is called {\sf Yoneda embedding};
and the Yoneda lemma gives an isomorphism
\begin{equation}
\label{eq:Yoneda}
\Hom_{\bD(\cT)}(\bh_{\cT,t},M) \cong \rH^0(M(t))
\end{equation} 
for any dg-module~$M$ over~$\cT$. 
In particular, the Yoneda embedding is fully faithful.

The thick triangulated subcategory of~$\bD(\cT)$ generated by representable dg-modules is called 
{\sf the subcategory of perfect dg-modules}; we denote it by~$\Dp(\cT) \subset \bD(\cT)$;
it coincides with the subcategory of compact objects of~$\bD(\cT)$.
If~$\cT$ is triangulated and idempotent complete (which we will usually assume), 
the Yoneda embedding gives an equivalence~$\cT \simeq \Dp(\cT)$;
in any case representable dg-modules in~$\bD(\cT)$ form a set of compact generators for~$\bD(\cT)$. 
The following converse result will be used later.

\begin{lemma}[{\cite[Proof of Corollary~2.3]{Neeman1}}]
\label{lem:dqc-dp}
If~$\hcT$ is a cocomplete dg-enhanced triangulated category 
and~$\cT \subset \hcT$ is a small triangulated subcategory 
such that the objects of~$\cT$ form a set of compact generators for~$\hcT$
then the functor
\begin{equation*}
\hcT \to \bD(\cT),
\qquad 
T \mapsto \RHom_{\hcT}(-,T)
\end{equation*}
is an equivalence.
\end{lemma}

The next lemma will be used to transfer representability results for cohomological functors 
to representability results for dg-modules.

\begin{lemma}
\label{lem:representability}
If~$\cT$ is a dg-enhanced triangulated category and~$M \in \bD(\cT)$ is a dg-module 
such that the cohomological functor~$\rH^0(M(-))$
is represented by an object~$t_M \in \cT$
then the dg-module~$M$ is represented by~$t_M$ as well.
\end{lemma}

\begin{proof}
We have by assumption an isomorphism of cohomological functors
\begin{equation*}
\Hom_\cT(-, t_M) \cong \rH^0(M(-)).
\end{equation*}
The identity~$\id_{t_M} \in \Hom_\cT(t_M,t_M)$ corresponds to a class~$\alpha \in \rH^0(M(t_M))$;
we denote by~$\tilde\alpha \colon \bh_{\cT,t_M} \to M$ the morphism in~$\bD(\cT)$
corresponding to~$\alpha$ by~\eqref{eq:Yoneda}.
Evaluating~$\tilde\alpha$ on an object~$t \in \cT$ we obtain the isomorphism
\begin{equation*}
\rH^0(\tilde\alpha(t)) \colon \rH^0(\bh_{\cT,t_M}(t)) =
\Hom_\cT(t,t_M) \xrightiso{}
\rH^0(M(t)).
\end{equation*}
Applying this to shifts of~$t$ 
we deduce that the same is true for~$\rH^i$ for all~$i \in \ZZ$,
hence~$\tilde\alpha(t)$ is a quasi-isomorphism, and hence~$M$ is represented by~$t_M$.
\end{proof}

Let~$\varphi \colon \cT_1 \to \cT_2$ be a dg-enhanced triangulated functor between essentially small dg-enhanced triangulated categories.
We consider the induced restriction functor on the derived categories
\begin{equation}
\label{eq:def-res}
\Res(\varphi)  \colon \bD(\cT_2) \to \bD(\cT_1),
\qquad 
\Res(\varphi)(M)(t_1) \coloneqq M(\varphi(t_1)),
\end{equation}
where~$M \in \bD(\cT_2)$ and~$t_1 \in \cT_1$.
Evidently, this functor is continuous.

The following property is obvious and well-known, so we omit the proof.

\begin{lemma}
\label{lem:res}
The operation~$\cT \mapsto \bD(\cT)$, $\varphi \mapsto \Res(\varphi)$ is a contravariant pseudofunctor
from the $2$-category of small triangulated dg-categories 
to the $2$-category of cocomplete triangulated dg-categories and continuous dg-functors.
\end{lemma}

The restriction functor has a left adjoint, which is also continuous (see, e.g., \cite[\S3.9]{KL}).

\begin{lemma}
\label{lem:res-ind}
If~$\varphi \colon \cT_1 \to \cT_2$ is a dg-enhanced triangulated functor, there is a continuous functor
\begin{equation*}
\Ind(\varphi) \colon \bD(\cT_1) \to \bD(\cT_2)
\end{equation*}
which is left adjoint to~$\Res(\varphi)$ and is compatible with~$\varphi$ and the Yoneda embeddings in the sense
that the following diagram commutes
\begin{equation}
\label{eq:ind-pis}
\vcenter{
\xymatrix@C=4em{
\cT_1 \ar[d]_{\bh_{\cT_1}} \ar[r]^\varphi & \cT_2 \ar[d]^{\bh_{\cT_2}} \\
\bD(\cT_1) \ar[r]^{\Ind(\varphi)} & \bD(\cT_2) \\
}}
\end{equation}
\end{lemma}

Finally, recall from~\cite[\S3.3]{Orlov-NC} 
the following properties of a dg-enhanced triangulated category~$\cT$:
\begin{itemize}
\item 
$\cT$ is {\sf proper}
if~$\RHom_\cT(t_1,t_2) \in \Dp(\kk)$ for any~$t_1,t_2 \in \cT$;
\item 
$\cT$ is {\sf smooth} 
if the diagonal bimodule~$\Delta_\cT$ is perfect;
\item 
$\cT$ is {\sf regular} 
if it has a \emph{strong generator}, i.e., an object~$t_0 \in \cT$ 
such that any~$t \in \cT$ can be obtained from~$t_0$ by shifts, direct sums, direct summands,
and a uniformly bounded number of cones.
\end{itemize}

Note that~$\cT$ is proper, smooth, or regular
if and only if the same properties hold for~$\cT^\opp$.
Note also that while regularity is an absolute property, smoothness and properness are relative with respect to the base field.
Finally, recall the following useful implication:

\begin{lemma}[{\cite[Lemma~3.5, 3.6]{Lunts}}]
\label{lem:smooth-regular}
If a dg-enhanced triangulated category~$\cT$ is smooth, it is regular.
\end{lemma}

\section{Homologically finite-dimensional objects}
\label{sec:hfd-objects}

In this section we introduce and develop the machinery of homologically finite-dimensional objects,
which is used as the main technical tool in the rest of the paper and define the reflexivity property.
Throughout this section~$\cT$ is an essentially small dg-enhanced idempotent complete $\kk$-linear triangulated category
over an arbitrary field~$\kk$.

\subsection{Finite-dimensional dg-modules}
\label{ss:hf}

We start with the main definition.

\begin{definition}
\label{def:rhfd-lhf}
A right dg-module $M \in \bD(\cT)$ over a small dg-enhanced triangulated category~$\cT$ 
is {\sf homologically finite-dimensional} if~$M(t) \in \Db(\kk) = \Dp(\kk)$ for any~$t \in \cT$.
We denote by
\begin{equation*}
\Dfd\cT \subset \bD(\cT)
\end{equation*}
the subcategory of all {\sf homologically finite-dimensional} right dg-modules over~$\cT$.
\end{definition}

Note that~$\Dfd{\cT}$ is a dg-enhanced essentially small idempotent complete triangulated subcategory.

\begin{lemma}
\label{lem:res-shriek}
If~$\varphi \colon \cT_1 \to \cT_2$ is an exact functor between small dg-enhanced triangulated categories then
\begin{equation*}
\Res(\varphi)(\Dfd{\cT_2}) \subset \Dfd{\cT_1}.
\end{equation*}
\end{lemma}

\begin{proof}
If~$t_1 \in \cT_1$ and~$M \in \Dfd{\cT_2}$ then~$\Res(\varphi)(M)(t_1) = M(\varphi(t_1)) \in \Db(\kk)$
because~$\varphi(t_1) \in \cT_2$.
\end{proof}

It follows from Lemma~\ref{lem:res-shriek} that the functor~$\Res(\varphi)$ induces a functor
\begin{equation}
\label{eq:shriek}
\varphi^! \coloneqq \Res(\varphi)\vert_{\Dfd{\cT_2}} \colon \Dfd{\cT_2} \to \Dfd{\cT_1}.
\end{equation}
This simple observation when combined with Lemma~\ref{lem:res} and Lemma~\ref{lem:res-ind} has the following corollaries.

\begin{corollary}
\label{cor:pseudofunctor}
The operation~$\cT \mapsto \Dfd{\cT}$, $\varphi \mapsto \varphi^!$ defines a contravariant pseudofunctor
from the $2$-category of small triangulated dg-categories to the $2$-category of small triangulated dg-categories.
\end{corollary}

\begin{corollary}
\label{cor:shriek-adjoint}
If~$\varphi \colon \cT_1 \to \cT_2$ is a dg-enhanced triangulated functor, 
the functor~$\varphi^! \colon \Dfd{\cT_2} \to \Dfd{\cT_1}$ defined in~\eqref{eq:shriek}
is its right adjoint in the sense that there is a functorial isomorphism
\begin{equation*}
\Hom_{\bD(\cT_2)}(\bh_{\cT_2,\varphi(t_1)}, M_2) \cong \Hom_{\bD(\cT_1)}(\bh_{\cT_1,t_1}, \varphi^!M_2)
\end{equation*}
for any~$t_1 \in \cT_1$, $M_2 \in \Dfd{\cT_2}$.
\end{corollary}

\begin{proof} 
Using~\eqref{eq:Yoneda}, \eqref{eq:def-res}, and~\eqref{eq:shriek}, we obtain
\begin{equation*}
\Hom_{\bD(\cT_2)}(\bh_{\cT_2,\varphi(t_1)}, M_2) \cong 
\rH^0(M_2(\varphi(t_1))) \cong
\rH^0(\varphi^!M_2(t_1)) \cong
\Hom_{\bD(\cT_1)}(\bh_{\cT_1,t_1}, \varphi^!M_2)
\end{equation*}
which gives the required adjunction.
\end{proof}

As we will see later, besides the category~$\Dfd{\cT}$ of finite-dimensional right dg-modules,
it is useful to consider the category~$\Dfd{\cT^\opp}$ of finite-dimensional left dg-modules over~$\cT$; 
these two categories are related by the dualization operation.

\begin{lemma}
\label{lemma:dual-kk}
The dualization operation on complexes of $\kk$-vector spaces~$V \mapsto V^\vee$ 
induces a dg-enhanced faithful triangulated functor 
\begin{equation}
\label{eq:duk}
\duk \colon \bD(\cT) \to \bD(\cT^\opp)^\opp,
\qquad 
\quad
\duk(M)(t) \coloneqq M(t)^\vee.
\end{equation}
It preserves homologically finite-dimensional dg-modules and induces a commutative diagram
\begin{equation}
\label{eq:diagram-rhfd-lhf}
\vcenter{\xymatrix@C=3em{
\bD(\cT) \ar[r]^-\duk & 
\bD(\cT^\opp)^\opp
\\
\Dfd{\cT}
\ar[r]_-\duk^-{\simeq} \ar@{^{(}->}[u] & 
\Dfd{\cT^\opp}^\opp, \ar@{^{(}->}[u] 
}}
\end{equation} 
where the bottom arrow is an equivalence whose inverse is also given by~$\duk$.
\end{lemma}

\begin{proof}
The lemma easily follows from the fact that the dualization functor on vector spaces is faithful,
and is an equivalence when restricted to finite-dimensional spaces.
\end{proof}

Note that a functor~$\varphi \colon \cT_1 \to \cT_2$ can be also thought of 
as a functor~$\varphi^\opp \colon \cT_1^\opp \to \cT_2^\opp$ between the opposite categories.
From this we obtain~$\Res(\varphi^\opp) \colon \bD(\cT_2^\opp) \to \bD(\cT_1^\opp)$ 
and~$(\varphi^\opp)^! \colon \Dfd{\cT_2^\opp} \to \Dfd{\cT_1^\opp}$ by restriction,
and passing to the opposite again we obtain the functor
\begin{equation}
\label{eq:star-body}
\varphi^* \coloneqq 
((\varphi^\opp)^!)^\opp \colon \Dfd{\cT_2^\opp}^\opp \to \Dfd{\cT_1^\opp}^\opp.
\end{equation}
The following result is immediate.

\begin{lemma}
\label{lem:star-adjunction}
The functor~$\varphi^* \colon \Dfd{\cT_2^\opp}^\opp \to \Dfd{\cT_1^\opp}^\opp$ 
is left adjoint to the functor~$\varphi \colon \cT_1 \to \cT_2$, i.e.,
\begin{equation*}
\Hom_{\bD(\cT_1^\opp)^\opp}(\varphi^*M_2, \bh_{\cT_1^\opp,t_1}) \cong
\Hom_{\bD(\cT_2^\opp)^\opp}(M_2, \bh_{\cT_2^\opp,\varphi^\opp(t_1)})
\end{equation*}
for any~$t_1 \in \cT_1$, $M_2 \in \Dfd{\cT_2^\opp}^\opp$.
\end{lemma}

\begin{proof}
Follows immediately from Corollary~\ref{cor:shriek-adjoint} applied to~$\varphi^\opp \colon \cT_1^\opp \to \cT_2^\opp$
since passing to the opposite category reverses the direction of morphisms. 
\end{proof}

The adjunction of~$\varphi^*$ and~$\varphi$ looks a bit messy, 
but it will become more transparent for functors between hfd-closed categories (defined in~\S\ref{sec:hfd-gor}),
see Proposition~\ref{prop:adjunctions}.

The pseudofunctors~$\bD(-)$ and~$\Dfd{-}$, as any other pseudofunctors between 2-categories of triangulated categories,
are compatible with adjunctions and semiorthogonal decompositions.
Note, however, that they invert the order of the components (because they are contravariant).

\begin{lemma}
\label{lem:sod-rhfd-lhf}
If $\cT = \langle \cA, \cB \rangle$ is a semiorthogonal decomposition, there are semiorthogonal decompositions
\begin{equation}
\label{eq:sod-rhfd-lhf}
\bD(\cT) = \langle \bD(\cB), \bD(\cA) \rangle
\qquad\text{and}\qquad 
\Dfd{\cT} = \langle \Dfd{\cB}, \Dfd{\cA} \rangle.
\end{equation}
Moreover, as subcategories of~$\bD(\cT)$ and~$\Dfd{\cT}$ their components 
consist of \textup(homologically finite-dimen\-sional\textup) dg-modules over~$\cT$
that vanish on the subcategories~$\cA \subset \cT$ and~$\cB \subset \cT$, respectively.
\end{lemma}
\begin{proof}
If~$\iota_{\cA} \colon \cA \to \cT$ and~$\iota_{\cB} \colon \cB \to \cT$ are the embedding functors,
and~$\pi_{\cA} \colon \cT \to \cA$, $\pi_{\cB} \colon \cT \to \cB$ are the projection functors
with respect to~$\cT = \langle \cA, \cB \rangle$,
we have the adjunction relations for~$\iota_{\cA}$ and~$\pi_{\cA}$, i.e.,
the unit and counit morphisms~$\id_\cT \to \iota_{\cA} \circ \pi_{\cA}$ and~$\pi_{\cA} \circ \iota_{\cA} \to \id_\cA$
such that the compositions
\begin{equation}
\label{eq:adjunction-relations}
\pi_{\cA} \to \pi_{\cA} \circ \iota_{\cA} \circ \pi_{\cA} \to \pi_{\cA}
\qquad\text{and}\qquad
\iota_{\cA} \to \iota_{\cA} \circ \pi_{\cA} \circ \iota_{\cA} \to \iota_{\cA}
\end{equation}
are the identity morphisms.
We also have analogous relations for~$\pi_{\cB}$ and~$\iota_{\cB}$
(with the change that now~$\pi_{\cB}$ is right adjoint to~$\iota_{\cB}$),
and semiorthogonal decomposition relations 
\begin{equation}
\label{eq:sod-relations}
\pi_{\cA} \circ \iota_{\cA} = \id_\cA,\quad
\pi_{\cB} \circ \iota_{\cB} = \id_\cB,\quad
\pi_{\cA} \circ \iota_{\cB} = 0,\quad
\pi_{\cB} \circ \iota_{\cA} = 0.
\end{equation}
After application of the pseudofunctor~$\bD(-)$ we obtain continuous functors
\begin{equation*}
\xymatrix@1@C=5em{
  \bD(\cA)\ \ar@<.5ex>[r]^-{\Res(\pi_\cA)} &
\ \bD(\cT)\ \ar@<.5ex>[l]^-{\Res(\iota_\cA)} \ar@<.5ex>[r]^-{\Res(\iota_\cB)} &
\ \bD(\cB)  \ar@<.5ex>[l]^-{\Res(\pi_\cB)}
},
\end{equation*}
and it follows from Lemma~\ref{lem:res} that these functors satisfy the analogous relations (with the direction of all arrows inverted):
applying~$\Res$ to~\eqref{eq:adjunction-relations} we deduce that~$\Res(\iota_{\cA})$ is right adjoint to~$\Res(\pi_{\cA})$,
and similarly~$\Res(\iota_{\cB})$ is left adjoint to~$\Res(\pi_{\cB})$;
and applying~$\Res$ to~\eqref{eq:sod-relations} we deduce the relations 
\begin{align*}
\Res(\iota_{\cA}) \circ \Res(\pi_{\cA}) &= \id_{\bD(\cA)}, 
&
\Res(\iota_{\cA}) \circ \Res(\pi_{\cB}) &= 0,
\\ 
\Res(\iota_{\cB}) \circ \Res(\pi_{\cB}) &= \id_{\bD(\cB)}, 
&
\Res(\iota_{\cB}) \circ \Res(\pi_{\cA}) &= 0,
\end{align*}
i.e., semiorthogonality relations in~$\bD(\cT)$.

Next, we show that~$N \in \bD(\cT)$ is in the image of~$\Res(\pi_\cA)$ if and only if it vanishes on~$\cB \subset \cT$.
Indeed,
\begin{equation*}
\Res(\pi_\cA)(M)(\iota_\cB(b)) = M(\pi_\cA(\iota_\cB(b))) = M(0) = 0
\end{equation*}
for~$M \in \bD(\cA)$, $b \in \cB$ gives one of the inclusions. 
Conversely, if~$N$ vanishes on~$\cB$
then the cone of the counit of adjunction~$\Res(\pi_\cA)(\Res(\iota_\cA)(N)) \to N$ vanishes on~~$\cB$.
But it also vanishes on~$\cA$ by the relations proved above, 
and since~$\cT$ is generated by~$\cA$ and~$\cB$, it vanishes on~$\cT$, hence it is zero, 
and hence~$N$ is in the image of~$\Res(\pi_\cA)$.
This proves that~$\Res(\pi_\cA)(\bD(\cA)) = \Ker(\Res(\iota_\cB))$.
The same argument proves~$\Res(\pi_\cB)(\bD(\cB)) = \Ker(\Res(\iota_\cA))$.

Now we show that~$\Res(\pi_\cA)(\bD(\cA))$ and~$\Res(\pi_\cB)(\bD(\cB))$ generate~$\bD(\cT)$.
Indeed, since the first subcategory is right admissible and the corresponding projection functor is~$\Res(\iota_\cA)$,
this follows from the equality~$\Res(\pi_\cB)(\bD(\cB)) = \Ker(\Res(\iota_\cA))$ proved above.

The same argument (with~$\Res(\iota_{\cA})$, $\Res(\pi_{\cA})$, $\Res(\iota_{\cB})$, and~$\Res(\pi_{\cB})$ 
replaced by~$\iota_\cA^!$, $\pi_\cA^!$, $\iota_\cB^!$, and~$\pi_\cB^!$, see~\eqref{eq:shriek}) 
proves the semiorthogonal decomposition for~$\Dfd{\cT}$.
\end{proof}

Finally, we discuss the consequences of standard homological properties of~$\cT$ 
(recalled in~\S\ref{sec:prelim}) for the category~$\Dfd{\cT}$.
When we compare~$\cT$ to~$\Dfd{\cT}$, we always consider both as subcategories of~$\bD(\cT)$.

\begin{lemma}
\label{lemma:sp-hfd}
Let~$\cT$ be an essentially small dg-enhanced idempotent complete triangulated category.
\begin{enumerate}[label={\textup{(\roman*)}}]
\item
\label{item:hfd-proper}
$\cT$ is proper if and only if~$\cT \subset \Dfd{\cT} \subset \bD(\cT)$. 
\item
\label{item:hfd-smooth}
If~$\cT$ is smooth 
then~$\Dfd{\cT} \subset \cT \subset \bD(\cT)$.
\item
\label{item:hfd-regular-proper}
If~$\cT$ is regular and proper then~$\Dfd{\cT} = \cT$.
\end{enumerate}
In particular,  if~$\cT$ is smooth and proper then~$\Dfd{\cT} = \cT$.
\end{lemma}

\begin{proof}
Assertion~\ref{item:hfd-proper} is obvious. 
To prove~\ref{item:hfd-smooth}
recall that any dg-module~$B$ over~$\cT_1^\opp \otimes \cT_2$ gives
rise to the derived tensor product functor~$\Phi_B \colon \Du(\cT_1) \to \Du(\cT_2)$ 
(see~\cite[\S6.1]{Keller} or~\cite[\S\S3.5--3.6]{KL}).
Moreover, for a representable dg-module~$B = \bh_{\cT_1^\opp \otimes \cT_2 ,t_1 \boxtimes t_2}$ 
and any~$M \in \Du(\cT_1)$ we have 
\begin{equation*}
\Phi_B(M) \cong M(t_1) \otimes \bh_{\cT_2, t_2} \in \Du(\cT_2).
\end{equation*}
It follows that, for a representable dg-module~$B$, 
we have~$\Phi_B(\Dfd{\cT_1}) \subset \Dp(\cT_2)$,
hence the same holds for any perfect dg-module~$B$.
When~$\cT_1 = \cT_2 = \cT$ is smooth, 
the diagonal bimodule~$\Delta_\cT$ is a perfect dg-module over~$\cT^\opp \otimes \cT$,
hence
\begin{equation*}
\Phi_{\Delta_\cT}(\Dfd{\cT}) \subset \Dp(\cT).
\end{equation*}
On the other hand, $\Phi_{\Delta_\cT}$ is the identity functor, so $\Dfd{\cT} \subset \Dp(\cT)$.
Finally, $\Dp(\cT) = \cT$ since~$\cT$ is triangulated and idempotent complete.

\ref{item:hfd-regular-proper}
Follows from~\cite[Theorem~1.3]{Bondal-vdB} and Lemma~\ref{lem:representability}.

The last assertion is a combination of~\ref{item:hfd-proper} and~\ref{item:hfd-smooth};
it also follows from~\ref{item:hfd-regular-proper} combined with Lemma~\ref{lem:smooth-regular}.
\end{proof}

The following example, communicated to us by Dima Orlov, shows that 
for a regular triangulated category in general~$\Dfd{\cT} \not\subset \cT$.
Indeed, let~$\cT$ be the perfect derived category of representations of the quiver with two vertices 
and infinitely many arrows from the first vertex to the second.
This category~$\cT$ is generated by an exceptional pair (formed by the projective representations), hence it is regular.
Obviously, the simple representations belong to~$\Dfd{\cT}$, 
but the minimal projective resolution for one of them involves one projective representation with infinite multiplicity, 
hence this simple does not lie in~$\cT$.

We finish this subsection with a warning.

\begin{remark}
\label{rem:warning-1}
Assume~$\cT$ is proper, so that~$\cT \subset \Dfd{\cT} \subset \bD(\cT)$.
The embedding~$\iota \colon \cT \hookrightarrow \Dfd{\cT}$ 
induces an embedding~$\Ind(\iota) \colon \bD(\cT) \hookrightarrow \bD(\Dfd{\cT})$,
so we have two embeddings
\begin{equation*}
\Ind(\iota)\vert_{\Dfd{\cT}} \colon \Dfd{\cT} \hookrightarrow \bD(\Dfd{\cT})
\qquad\text{and}\qquad 
\bh_{\Dfd{\cT}} \colon \Dfd{\cT} \hookrightarrow \bD(\Dfd{\cT}).
\end{equation*}
Somewhat unexpectedly, these embeddings are different!
They do, however, coincide when restricted to the subcategory~$\cT$ by commutativity of the diagram~\eqref{eq:ind-pis}.
\end{remark}

\subsection{Reflexivity}
\label{ss:refelxivity}

In this section we investigate what happens when we apply the pseudofunctor~$\Dfd{-}$ twice.
For each object~$t \in \cT$ we define
\begin{align}
\label{eq:def-ev}
\ev_{\cT,t}(M) &\coloneqq M(t),
&& 
M \in \Dfd{\cT},
\\
\label{eq:def-coev}
\coev_{\cT,t}(M) &\coloneqq M(t)^\vee,
&& 
M \in \Dfd{\cT}.
\end{align}
The first, $\ev_{\cT,t}$, is a left homologically finite-dimensional dg-module over~$\Dfd{\cT}$, 
and the second, $\coev_{\cT,t}$, is a right homologically finite-dimensional dg-module over~$\Dfd{\cT}$.
Moreover, these dg-modules are functorial in~$t$, the first is contravariant and the second is covariant,
hence we have dg-functors
\begin{align}
\label{eq:def-ev-functor}
\ev_\cT &\colon \cT^\opp \to  \Dfd{\Dfd{\cT}^\opp},
&& 
t \mapsto \ev_{\cT,t}
\\
\label{eq:def-coev-functor}
\coev_\cT &\colon \cT\hphantom{{}^\opp} \to  \Dfd{\Dfd{\cT}},
&& 
t \mapsto \coev_{\cT,t}.
\end{align}
We can also consider the dg-functor~$\coev_{\cT^\opp} \colon \cT^\opp \to \Dfd{\Dfd{\cT^\opp}}$.

\begin{lemma}
\label{lem:reflexivity-opp}
One of the functors~$\coev_\cT$, $\ev_\cT$, or~$\coev_{\cT^\opp}$ is fully faithful, or essentially surjective, or is an equivalence
if and only if the other two have the same property.
\end{lemma}

\begin{proof}
The definitions easily imply the commutativity of the diagram
\begin{equation*}
\xymatrix@C=7em{
&
\cT^\opp \ar[dr]^{(\coev_\cT)^\opp}
\ar[d]^{\ev_\cT}
\ar[dl]_{\coev_{\cT^\opp}} 
\\
\Dfd{\Dfd{\cT^\opp}} \ar[r]_{(\mathbf{d}_\Bbbk^\opp)^!} &
\Dfd{\Dfd{\cT}^\opp} &
\Dfd{\Dfd{\cT}}^\opp \ar[l]^{\mathbf{d}_\Bbbk^\opp},
}
\end{equation*}
where the right bottom arrow is induced by the dualization functor for the category~$\Dfd{\cT}$,
and the left bottom arrow is induced by~$\mathbf{d}_\Bbbk^\opp \colon \Dfd{\cT}^\opp \to \Dfd{\cT^\opp}$.
Now the lemma follows from the fact that the bottom arrows are equivalences by Lemma~\ref{lemma:dual-kk}.
\end{proof}

\begin{definition}
\label{def:reflexivity}
We say that a small dg-enhanced triangulated category~$\cT$ is {\sf reflexive}
if the dg-functor~$\coev_\cT$ in~\eqref{eq:def-coev-functor} is an equivalence.
\end{definition}

We could alternatively use the dg-functor~$\ev_\cT$ in this definition,
but Lemma~\ref{lem:reflexivity-opp} shows that this would be an equivalent definition.
We prefer using~$\coev_\cT$, because its definition does not require considering opposite categories.
Note also that by Lemma~\ref{lem:reflexivity-opp} a category~$\cT$ is reflexive if and only if so is~$\cT^\opp$.

Later on (see~\S\ref{sec:geometry} and~\S\ref{sec:algebra}) we will see that many categories 
in geometry and algebra are reflexive.
Here is an example of a
non-reflexive category.

\begin{example}
Let~$\cT = \Dp(\kk\langle x,y \rangle/(xy - yx = 1))$ be
the perfect derived category of the Weyl algebra over a field~$\kk$ of characteristic zero.
If~$M$ is a homologically finite-dimensional dg-module 
then each of its cohomology is a finite-dimensional module over the Weyl algebra, 
hence zero (because the trace of a commutator of two operators on a finite-dimensional vector space is zero), and so~$M = 0$.
Thus, in this example~$\Dfd{\cT} = 0$, hence~$\Dfd{\Dfd{\cT}} = 0 \not\simeq{} \cT$.
\end{example}

Now we deduce some formal properties of reflexivity.

\begin{lemma}
\label{lem:lhfd-lhfrefl}
If~$\cT$ is reflexive, then~$\Dfd{\cT}$ is reflexive.
\end{lemma}
\begin{proof}
Consider the following functors
\begin{equation*}
\coev_{\Dfd{\cT}} \colon \Dfd{\cT} \to \Dfd{\Dfd{\Dfd{\cT}}},
\qquad
(\coev_{\cT})^! \colon \Dfd{\Dfd{\Dfd{\cT}}} \to \Dfd{\cT},
\end{equation*}
(the first is defined by~\eqref{eq:def-coev-functor} and the second is defined by~\eqref{eq:shriek} and~\eqref{eq:def-coev-functor}).
For~$M \in \Dfd{\cT}$, $t \in \cT$ we have
\begin{equation*}
(\coev_\cT)^! (\coev_{\Dfd{\cT}}(M))(t) = 
\coev_{\Dfd{\cT}}(M)(\coev_{\cT,t}) = 
\coev_{\cT,t}(M)^\vee = 
M(t)^{\vee\vee} \cong 
M(t),
\end{equation*}
which means that~$(\coev_\cT)^! \circ \coev_{\Dfd{\cT}} \cong \id_{\Dfd{\cT}}$. 
On the other hand, since~$\cT$ is reflexive, $\coev_{\cT}$ is an equivalence, 
hence~$(\coev_{\cT})^!$ is also an equivalence by Corollary~\ref{cor:pseudofunctor}.
Then the above isomorphism implies that~$\coev_{\Dfd{\cT}}$ is also an equivalence, and hence~$\Dfd{\cT}$ is reflexive.
\end{proof}

We now state a useful criterion for a proper category to be reflexive.
Recall Lemma~\ref{lemma:sp-hfd}\ref{item:hfd-proper}.

\begin{lemma}
\label{lem:reflexivity-criterion}
Assume that~$\cT$ is proper, so that~$\cT \subset \Dfd{\cT}$.
Then the functors~$\ev_\cT$ and~$\coev_\cT$ are fully faithful.
In particular, if~$\ev_\cT$ or~$\coev_\cT$ is essentially surjective, $\cT$ and~$\Dfd{\cT}$ are reflexive.
\end{lemma}

\begin{proof}
Consider the diagram
\begin{equation}
\label{eq:hh-ev}
\vcenter{\xymatrix@C=5em{
\cT^\opp \ar[r]^-{\ev_\cT} \ar[d]_{(\bh_\cT)^\opp} & 
\Dfd{\Dfd{\cT}^\opp} \ar@{^{(}->}[d]
\\
\Dfd{\cT}^\opp \ar[r]^-{\bh_{\Dfd{\cT}^\opp}} &
\bD(\Dfd{\cT}^\opp)
}}
\end{equation}
The Yoneda functor of~$\cT$ takes values in homologically finite-dimensional modules because~$\cT$ is proper,
hence the left vertical arrow makes sense.
The composition of the left vertical and bottom horizontal arrow takes~$t \in \cT$ 
to the left dg-module over~$\Dfd{\cT}$ whose value on~$M \in \Dfd{\cT}$ is
\begin{equation*}
\RHom_{\Dfd{\cT}^\opp}(M,\bh_{\cT,t}) =
\RHom_{\Dfd{\cT}}(\bh_{\cT,t},M)  \cong 
M(t),
\end{equation*}
where the isomorphism follows from the Yoneda lemma.
By~\eqref{eq:def-ev}
the diagram is commutative.
Now note that the left vertical and bottom horizontal arrows are fully faithful by the Yoneda lemma,
hence the functor~$\ev_\cT$ is fully faithful, hence so is~$\coev_\cT$ by Lemma~\ref{lem:reflexivity-opp}.

Now if either~$\ev_\cT$ or~$\coev_\cT$ is essentially surjective, 
each of them is an equivalence (see Lemma~\ref{lem:reflexivity-opp}), 
hence~$\cT$ is reflexive.
Finally, $\Dfd{\cT}$ is reflexive by Lemma~\ref{lem:lhfd-lhfrefl}.
\end{proof}

The coevaluation is a morphism of pseudofunctors from the identity to the square of~$\Dfd{-}$.

\begin{lemma}
\label{lem:coev-commutes}
If~$\varphi \colon \cT_1 \to \cT_2$ is any dg-enhanced functor,  the diagram
\begin{equation}
\label{eq:coev-commutes}
\vcenter{\xymatrix@C=5em{
\cT_1 \ar[d]_{\coev_{\cT_1}} \ar[r]^\varphi &
\cT_2 \ar[d]^{\coev_{\cT_2}} 
\\
\Dfd{\Dfd{\cT_1}} \ar[r]^{(\varphi^!)^!} &
\Dfd{\Dfd{\cT_2}}.
}}
\end{equation}
commutes.
In particular, if~$\varphi$, $(\varphi^!)^!$, and~$\coev_{\cT_2}$ are fully faithful then so is~$\coev_{\cT_1}$.
\end{lemma}

\begin{proof}
For all~$t_1 \in \cT_1$ and~$M_2 \in \Dfd{\cT_2}$ we have 
(using the definitions~\eqref{eq:shriek} and~\eqref{eq:def-coev})
a chain of equalities
\begin{equation*}
(\varphi^!)^!(\coev_{\cT_1}(t_1))(M_2) =
\coev_{\cT_1}(t_1)(\varphi^!M_2) =
(\varphi^!M_2)(t_1)^\vee = 
M_2(\varphi(t_1))^\vee =
\coev_{\cT_2}(\varphi(t_1))(M_2),
\end{equation*}
which proves the commutativity.
The second claim is obvious.
\end{proof}

One could interpret commutativity of diagram~\eqref{eq:coev-commutes} 
as an equivalence between the functor categories 
from~$\cT_1$ to~$\cT_2$ and from~$\Dfd{\Dfd{\cT_1}}$ to~$\Dfd{\Dfd{\cT_2}}$
when~$\cT_1$ and~$\cT_2$ are reflexive.
We leave this to the interested reader, while proving only the following corollary
for the sets of isomorphisms classes of functors~$\cT_1 \to \cT_2$, that we denote~$[\cT_1, \cT_2]$.

\begin{corollary}
\label{cor:equivalence-Refl}
Given reflexive triangulated dg-enhanced categories~$\cT_1$, $\cT_2$ we have a natural bijection
\begin{equation}
\label{eq:refl-duality}
[\cT_1,\cT_2] \xrightiso{} [\Dfd{\cT_2},\Dfd{\cT_1}], \qquad
\varphi \mapsto \varphi^!
\end{equation}
In particular, every equivalence~$\Dfd{\cT_2} \xrightiso{} \Dfd{\cT_1}$
is induced by a unique \textup(up to isomorphism\textup) equivalence~$\cT_1 \xrightarrow{} \cT_2$.
\end{corollary}

\begin{proof}
We first show that if~$\cT_1$ and~$\cT_2$ are reflexive,
then we have a bijection 
\begin{equation}
\label{eq:refl-double-shrieck}
[\cT_1,\cT_2] \xrightiso{} [\Dfd{\Dfd{\cT_1}},
\Dfd{\Dfd{\cT_2}}], \quad
\varphi \mapsto (\varphi^!)^!
\end{equation}
Indeed, since~$\coev_{\cT_1}$ and~$\coev_{\cT_2}$ are equivalences, commutativity of~\eqref{eq:coev-commutes} means that 
\begin{equation*}
\varphi \cong \coev_{\cT_2}^{-1} \circ (\varphi^!)^! \circ \coev_{\cT_1}
\qquad\text{and}\qquad
(\varphi^!)^! \cong \coev_{\cT_2} \circ \varphi \circ \coev_{\cT_1}^{-1},
\end{equation*}
hence the map that takes~$\psi \colon \Dfd{\Dfd{\cT_1}} \xrightiso{} \Dfd{\Dfd{\cT_2}}$
to~$\varphi \coloneqq \coev_{\cT_2}^{-1} \circ \psi \circ \coev_{\cT_1}$
is inverse to~\eqref{eq:refl-double-shrieck}.
Note also that the above isomorphisms show that~$\varphi$ is an equivalence if and only if so is~$(\varphi^!)^!$.

Now, to prove that~\eqref{eq:refl-duality} is a bijection, we consider the composition
\begin{equation*}
[\cT_1,\cT_2] \to 
[\Dfd{\cT_2},\Dfd{\cT_1}] \to 
[\Dfd{\Dfd{\cT_1}},\Dfd{\Dfd{\cT_2}}] \to 
[\Dfd{\Dfd{\Dfd{\cT_2}}}, \Dfd{\Dfd{\Dfd{\cT_1}}}]
\end{equation*}
of maps~\eqref{eq:refl-duality}.
Note that the categories~$\cT_1$, $\cT_2$, $\Dfd{\cT_1}$, and~$\Dfd{\cT_2}$ are reflexive
(the first two by assumption, the last two by Lemma~\ref{lem:lhfd-lhfrefl}),
therefore, the pairwise compositions of arrows, which are the maps~\eqref{eq:refl-double-shrieck}, are bijections.
We conclude that the middle arrow is surjective (because the composition of the first two arrows is bijective)
and injective (because the composition of the last two arrows is bijective), hence it is bijective.
Therefore, \eqref{eq:refl-duality} is also a bijection.

Finally, if~$\varphi$ is an equivalence, then so is~$\varphi^!$ by Corollary~\ref{cor:pseudofunctor}.
Conversely, if~$\varphi^!$ is an equivalence then so is~$(\varphi^!)^!$, 
and then, as we have checked above, so is~$\varphi$.
Thus, \eqref{eq:refl-duality} induces a bijection on the sets of isomorphism classes of equivalences.
\end{proof}

One reason why reflexivity is important is the following theorem building up on Lemma~\ref{lem:sod-rhfd-lhf}.
Recall that~$\LAdm(\cT)$ and~$\RAdm(\cT)$ denote the sets 
of all left and right admissible subcategories of~$\cT$.
In the statement we think of a left or right admissible subcategory as a triple --- 
the subcategory, its embedding functor, and its appropriate (left or right) adjoint.

\begin{theorem}
\label{thm:lhfd-right-bijection}
If~$\cT$ is a reflexive small dg-enhanced triangulated category we have bijections
\begin{equation}
\label{eq:admissiblity-bijection}
\begin{aligned}
\LAdm(\cT) & \cong \RAdm(\Dfd{\cT}), \qquad\qquad&
\RAdm(\cT) & \cong \LAdm(\Dfd{\cT}), 
\\
(\cA,\iota,\pi_L) &\mapsto (\Dfd{\cA},\pi_L^!,\iota^!), &
(\cA,\iota,\pi_R) &\mapsto (\Dfd{\cA},\pi_R^!,\iota^!), 
\end{aligned}
\end{equation} 
where~$\iota \colon \cA \to \cT$ is the embedding and~$\pi_L, \pi_R \colon \cT \to \cA$ 
are its left and right adjoint functors, respectively.

Moreover, if~$\cA \subset \cT$ is left or right admissible then~$\cA$ is reflexive.

\end{theorem}
\begin{proof}
Assume~$\cT = \langle \cA, \cB \rangle$.
Then there are two semiorthogonal decompositions
\begin{equation}
\label{eq:DfDf-sod}
\Dfd{\Dfd{\cT}} = \langle \coev_\cT(\cA), \coev_\cT(\cB) \rangle
\qquad\text{and}\qquad
\Dfd{\Dfd{\cT}} = \langle \Dfd{\Dfd{\cA}}, \Dfd{\Dfd{\cB}} \rangle,
\end{equation}
where the first follows from reflexivity of~$\cT$ and the second from Lemma~\ref{lem:sod-rhfd-lhf} applied twice.
The embeddings of the components in the first decomposition are given by~$\coev_\cT \circ \iota_\cA$ and~$\coev_\cT \circ \iota_\cB$,
and in the second they are given by~$(\iota_\cA^!)^!$ and~$(\iota_\cB^!)^!$, respectively.
Now applying Lemma~\ref{lem:coev-commutes} (with~$\varphi = \iota_\cA$ and~$\varphi = \iota_\cB$)
we see that~$\coev_\cT \circ \iota_\cA \cong (\iota_\cA^!)^! \circ \coev_\cA$
and~$\coev_\cT \circ \iota_\cB \cong (\iota_\cB^!)^! \circ \coev_\cB$,
hence the components of the first decomposition are contained in the components of the second, 
hence the components coincide and, therefore, we have the equalities
\begin{equation*}
\coev_\cT(\cA) = \Dfd{\Dfd{\cA}},
\qquad
\coev_\cT(\cB) = \Dfd{\Dfd{\cB}},
\end{equation*}
as subcategories of~$\Dfd{\Dfd{\cT}}$.
It also follows that~\mbox{$\coev_\cA \colon \cA \to \Dfd{\Dfd{\cA}}$
is essentially surjective.}
Moreover, since~$\iota_\cA$ and~$(\iota_\cA^!)^!$ are fully faithful, the second part of Lemma~\ref{lem:coev-commutes} 
implies that~$\coev_\cA$ is an equivalence,
and the same argument works for~$\cB$.
Therefore, $\cA$ and~$\cB$ are reflexive.

Finally, since~$\coev_\cT$ is an equivalence, it follows that the maps~\eqref{eq:admissiblity-bijection} are injective,
and applying the same argument for the subcategories of~$\Dfd{\cT}$ 
(note that the latter is reflexive by Lemma~\ref{lem:lhfd-lhfrefl}),
we conclude that the maps~\eqref{eq:admissiblity-bijection} are surjective.
\end{proof}

The following immediate corollary is quite useful.
Recall that a triangulated category~$\cT$ is called {\sf indecomposable} 
if for any semiorthogonal decomposition~$\cT = \langle \cA, \cB \rangle$ 
one has~$\cA = 0$ or~$\cB = 0$.

\begin{corollary}
\label{cor:indecomposability}
A reflexive category~$\cT$ is indecomposable if and only if~$\Dfd{\cT}$ is indecomposable.
\end{corollary}

\section{HFD-closed and Gorenstein categories}
\label{sec:hfd-gor}

In this section we define hfd-closed and Gorenstein categories and discuss their properties.

\subsection{HFD-closed categories}
\label{ss:hfd-closed}

We start by introducing the following notion.

\begin{definition}
\label{def:hfd-closed}
A dg-enhanced idempotent complete triangulated category~$\cT$ is {\sf hfd-closed} if
\begin{equation*}
\Dfd{\cT} \subset \cT \subset \bD(\cT)
\qquad\text{and}\qquad 
\Dfd{\cT^\opp} \subset \cT^\opp \subset \bD(\cT^\opp),
\end{equation*}
where the second embeddings are given by the Yoneda functors.
In other words, $\cT$ is hfd-closed 
if and only if any homologically finite-dimensional right or left dg-module over~$\cT$ 
is representable or corepresentable by an object of~$\cT$, respectively.
\end{definition}

In what follows we often consider the following triangulated subcategories of~$\cT$:
\begin{equation}
\label{eq:ct-rhf-lhf}
\begin{aligned}
\cT^\rhf &= \{ t \in \cT \mid \RHom_\cT(t',t) \in \Db(\kk) \ \text{for any~$t' \in \cT$} \},
\\
\cT^\lhf &= \{ t \in \cT \mid \RHom_\cT(t,t') \in \Db(\kk) \ \text{for any~$t' \in \cT$} \},
\end{aligned}
\end{equation}
so that for an hfd-closed subcategory we have identifications
\begin{equation}
\label{eq:dfd-hfd-closed}
\Dfd{\cT} = \bh_\cT(\cT^\rhf)
\qquad\text{and}\qquad 
\Dfd{\cT^\opp} = \bh_{\cT^\opp}((\cT^\lhf)^\opp).
\end{equation}

The following is an immediate consequence  
of Lemma~\ref{lemma:sp-hfd}.

\begin{corollary}
\label{cor:sp-hfd}
A smooth dg-enhanced idempotent complete triangulated category~$\cT$ is hfd-closed.
If additionally~$\cT$ is proper, then~$\cT^\rhf = \cT = \cT^\lhf$.
Similarly, if~$\cT$ is regular, proper, and idempotent complete, it is hfd-closed with~$\cT^\rhf = \cT = \cT^\lhf$.
\end{corollary}

\begin{proof}
If~$\cT$ is smooth then~$\cT^\opp$ is also smooth, 
and we conclude from Lemma~\ref{lemma:sp-hfd}\ref{item:hfd-smooth} applied to~$\cT$ and~$\cT^\opp$
that~$\cT$ is hfd-closed.
If~$\cT$ is also proper then the same is true for~$\cT^\opp$, 
and Lemma~\ref{lemma:sp-hfd}\ref{item:hfd-proper} implies the required equalities.
If~$\cT$ is regular and proper, we apply Lemma~\ref{lemma:sp-hfd}\ref{item:hfd-regular-proper}.
\end{proof}

Using Lemma~\ref{lem:reflexivity-criterion} we can prove the following

\begin{lemma}
\label{lem:sp-reflexive}
If~$\cT$ is hfd-closed and proper the category~$\cT$ is reflexive.
In particular, if~$\cT$ is smooth, proper, and idempotent complete
\textup{(}or regular, proper, and idempotent complete\textup{)}, it is reflexive.
\end{lemma}

\begin{proof}
We use the setup of Lemma~\ref{lem:reflexivity-criterion}, and particularly diagram~\eqref{eq:hh-ev}.
The assumption means that the left vertical arrow in~\eqref{eq:hh-ev} is an equivalence,
and that the bottom horizontal arrow is an equivalence onto the subcategory~$\Dfd{\Dfd{\cT}^\opp}$,
the image of the right vertical arrow.
Therefore the top horizontal arrow is an equivalence, hence~$\cT$ is reflexive by Lemma~\ref{lem:reflexivity-opp}.
The second claim follows from Corollary~\ref{cor:sp-hfd}.
\end{proof}

As we promised in~\S\ref{ss:hf}, for hfd-closed categories we give a better statement about adjoint functors.
Recall that if~$\varphi \colon \cT_1 \to \cT_2$ is a dg-enhanced triangulated functor, 
the functors~$\varphi^!$ and~$\varphi^*$ have been 
defined in~\eqref{eq:shriek} and~\eqref{eq:star-body}, respectively.
Note also that identifications~\eqref{eq:dfd-hfd-closed} allow us to consider these functors 
as functors~$\varphi^! \colon \cT_2^\rhf \to \cT_1^\rhf$ and~$\varphi^* \colon \cT_2^\lhf \to \cT_1^\lhf$.

\begin{proposition}
\label{prop:adjunctions}
If~$\varphi \colon \cT_1 \to \cT_2$ is a dg-enhanced triangulated functor between hfd-closed triangulated categories,
the functors~$\varphi^! \colon \cT_2^\rhf \to \cT_1^\rhf$ and~$\varphi^* \colon \cT_2^\lhf \to \cT_1^\lhf$ 
are its right and left adjoints.
\end{proposition}

\begin{proof}
This follows from Corollary~\ref{cor:shriek-adjoint} and Lemma~\ref{lem:star-adjunction}
and the identifications~\eqref{eq:dfd-hfd-closed}. 
\end{proof}

It follows immediately from the definition~\eqref{eq:ct-rhf-lhf} that any dg-enhanced equivalence~$\cT_1 \xrightiso{} \cT_2$
induces equivalences~$\cT_1^\rhf \xrightiso{} \cT_2^\rhf$ and~$\cT_1^\lhf \xrightiso{} \cT_2^\lhf$.
The converse is also true.

\begin{lemma}
\label{lem:equivalence-extension}
If~$\cT_1$ and~$\cT_2$ are hfd-closed and reflexive, 
any equivalence~$\varphi \colon \cT_1^\rhf \to \cT_2^\rhf$ 
extends to a unique equivalence~$\tilde\varphi \colon \cT_1 \to \cT_2$ such that~$\tilde\varphi\vert_{\cT_1^\rhf} = \varphi$.
A similar statement for equivalences~$\cT_1^\lhf \to \cT_2^\lhf$ also holds.
\end{lemma}

\begin{proof}
First, consider any equivalence~$\psi \colon \cT_1 \to \cT_2$.
We claim that the diagram
\begin{equation}
\label{eq:diagram-extension}
\vcenter{\xymatrix@C=5em{
\cT_1^\rhf \ar@{^{(}->}[r] \ar[d]_{(\psi^{-1})^!} &
\cT_1 \ar@{^{(}->}[r]^{\bh_{\cT_1}} \ar[d]^\psi &
\bD(\cT_1) \ar[d]^{\Res(\psi^{-1})}
\\
\cT_2^\rhf \ar@{^{(}->}[r] &
\cT_2 \ar@{^{(}->}[r]^{\bh_{\cT_2}} &
\bD(\cT_2). 
}}
\end{equation}
is commutative. 
Indeed, the right square is commutative because
\begin{equation*}
\Res(\psi^{-1})(\bh_{\cT_1}(t_1))(t_2) =
(\bh_{\cT_1,t_1})(\psi^{-1}(t_2)) =
\Hom_{\cT_1}(\psi^{-1}(t_2), t_1) \cong
\Hom_{\cT_2}(t_2, \psi(t_1)) =
\bh_{\cT_2}(\psi(t_1))(t_2),
\end{equation*}
and the ambient rectangle is commutative by definition of~$(\psi^{-1})^!$ as~$\cT_i^\rhf = \Dfd{\cT_i}$;
since~$\bh_{\cT_2}$ is fully faithful it follows that the left square is commutative as well.
In other words, this proves that~$\psi$ extends~$(\psi^{-1})^!$.
Now, since by Corollary~\ref{cor:equivalence-Refl} any equivalence~$\varphi \colon \cT_1^\rhf \to \cT_2^\rhf$ 
is isomorphic to~$(\psi^{-1})^!$ for some equivalence~$\psi$, the existence of extension follows.

On the other hand, if~$\tilde\varphi \colon \cT_1 \to \cT_2$ is another equivalence extending~$\varphi \cong (\psi^{-1})^!$, 
we have isomorphisms
\begin{equation*}
(\tilde\varphi^{-1})^! \cong 
\tilde\varphi\vert_{\cT_1^\rhf} \cong
\varphi \cong
(\psi^{-1})^!
\end{equation*}
(the first isomorphism follows from commutativity of the left square in~\eqref{eq:diagram-extension} for~$\tilde\varphi$),
hence~$\tilde\varphi \cong \psi$, again by Corollary~\ref{cor:equivalence-Refl}.
This proves the uniqueness of extension.

The claim for the equivalences~$\cT_1^\lhf \to \cT_2^\lhf$ is proved analogously.
\end{proof}

If a category~$\cT$ is hfd-closed, the operation of Lemma~\ref{lem:sod-rhfd-lhf} 
(which is bijective when~$\cT$ is reflexive by Theorem~\ref{thm:lhfd-right-bijection}) can be simplified.
For the reader's convenience in what follows we use notation~$\cA_0$ for a subcategory of~$\cT^\rhf$ or~$\cT^\lhf$
to distinguish it from a subcategory~$\cA$ of~$\cT$.

\begin{proposition}
\label{prop:bijectionsubcat-hfd-closed}
Let~$\cT$ be 
hfd-closed.

\begin{renumerate}
\item
\label{it:reflexive-hfd-closed-bijection}
The maps of Lemma~\textup{\ref{lem:sod-rhfd-lhf}} take the form
\begin{align}
\notag
\LAdm(\cT) & {}\to{} \RAdm(\cT^\rhf),
\qquad\qquad &
\RAdm(\cT) & {}\to{} \LAdm(\cT^\lhf),
\\
\label{eq:admissiblity-bijection-23}
\cA & \mapsto \cA \cap \cT^\rhf, &
\cA & \mapsto \cA \cap \cT^\lhf.
\end{align}
Moreover, $\cA \cap \cT^\rhf = \cA^\rhf$ for~$\cA \in \LAdm(\cT)$ 
and~$\cA \cap \cT^\lhf = \cA^\lhf$ for~$\cA \in \RAdm(\cT)$.

\item
\label{it:reflexive-hfd-closed-bijection-2}
If~$\cT$ is reflexive, so that the maps~\eqref{eq:admissiblity-bijection-23} are bijective,
the inverse maps are given by \begin{equation}
\label{eq:admissiblity-bijection-inverse}
({}^{\perp_{\cT^\rhf}} \cA_0)^{\perp_{\cT}} \mapsfrom \cA_0, 
\qquad\qquad\qquad\qquad\qquad\qquad
{}^{\perp_{\cT}}(\cA_0^{\perp_{\cT^{\lhf}}})   \mapsfrom \cA_0, 
\end{equation} 
where the orthogonals are taken in $\cT$, $\cT^\rhf$, $\cT^\lhf$ as indicated.

\item
\label{it:reflexive-hfd-closed-ca-hf}
If~$\cA \subset \cT$ is admissible then~$\cA$ is also hfd-closed.
\end{renumerate}
\end{proposition}

\begin{proof}
\ref{it:reflexive-hfd-closed-bijection}
Let~$\cA \in \LAdm(\cT)$, so that~$\cT = \langle \cA, {}^{\perp_\cT}\cA \rangle$.
We know from Lemma~\ref{lem:sod-rhfd-lhf} that the image of~$\cA$ under the map~$\LAdm(\cT) \to \RAdm(\cT^\rhf)$
coincides with the subcategory of all objects in~$\cT^\rhf$ 
such that the corresponding representable dg-module vanishes on the orthogonal~${}^{\perp_\cT}\cA$;
the Yoneda lemma identifies this with the double orthogonal category~$({}^{\perp_\cT}\cA)^{\perp_{\cT^\rhf}}$.
On the other hand, we have
\begin{equation*}
({}^{\perp_\cT}\cA)^{\perp_{\cT^\rhf}} = 
({}^{\perp_\cT}\cA)^{\perp_{\cT}} \cap \cT^\rhf = 
\cA \cap \cT^\rhf,
\end{equation*}
where the first follows from~$\cT^\rhf \subset \cT$ 
and the second from~$({}^{\perp_\cT}\cA)^{\perp_{\cT}} = \cA$, which holds because~$\cA$ is left admissible.
Moreover, the embedding~$\cA \cap \cT^\rhf \subset \cA^\rhf$ is obvious and the opposite embedding 
easily follows from left admissibility of~$\cA$;
thus,~$\cA \cap \cT^\rhf = \cA^\rhf$.
The second bijection can be proved by the same argument or by applying the first bijection to~$\cT^\opp$.

\ref{it:reflexive-hfd-closed-bijection-2}
The argument of part~\ref{it:reflexive-hfd-closed-bijection} 
describes the image of~$\cA_0 \in \RAdm(\cT^\rhf)$ or of~$\cA_0 \in \LAdm(\cT^\rhf)$ 
under the maps of Theorem~\ref{thm:lhfd-right-bijection} as double orthogonals in~$\Dfd{\cT^\rhf}$.
But by reflexivity of~$\cT$ we have~$\Dfd{\cT^\rhf} = \cT$, 
and under this identification the double orthogonal descriptions take the form~\eqref{eq:admissiblity-bijection-inverse}.

\ref{it:reflexive-hfd-closed-ca-hf}
Assume~$\cA \in \LAdm(\cT)$ and let~$\pi_L \colon \cT \to \cA$ be the projection to~$\cA$ 
with respect to the semiorthogonal decomposition~$\cT = \langle \cA, {}^\perp\cA \rangle$.
If~$M \in \Dfd{\cA}$ then~$\pi_L^!(M) \in \Dfd{\cT}$, hence~$\pi_L^!(M)$ is representable, i.e.,
\begin{equation*}
M(\pi_L(-)) {} = \pi_L^!(M)(-) \cong \RHom_\cT(-,t)
\end{equation*}
for some~$t \in \cT^\rhf$.
Moreover, $\RHom_\cT({}^\perp\cA, t) = M(\pi_L({}^\perp\cA)) = M(0) = 0$,
hence~\mbox{$t \in ({}^\perp\cA)^\perp = \cA$}.
Thus, $M$ is representable by an object of~$\cA$, hence~$\Dfd{\cA} \subset \cA$.

If~$\cA \in \RAdm(\cT)$ the same argument 
applied to the semiorthogonal decomposition~$\cT^\opp = \langle \cA^\opp, (\cA^\perp)^\opp \rangle$,
obtained from~$\cT = \langle \cA^\perp, \cA \rangle$ by passing to opposite categories 
proves the embedding~$\Dfd{\cA^\opp} \subset \cA^\opp$.

Combining these two inclusions we conclude that~$\cA$ is hfd-closed.
\end{proof}

In the next corollary we show that if one of the components of a semiorthogonal decomposition of~$\cT$ is admissible,
one can obtain a semiorthogonal decomposition of~$\cT^\lhf$ or~$\cT^\rhf$ simply by taking intersections.
Note that this is \emph{not the same} semiorthogonal decomposition as in Lemma~\ref{lem:sod-rhfd-lhf},
in particular, the order of its components is not inverted;
in fact, the embedding of the first component in the first decomposition 
and of the second component in the second decomposition below
is different from the corresponding embedding in Lemma~\ref{lem:sod-rhfd-lhf}.

\begin{corollary}
\label{cor:reflexive-hfd-closed-ct-hf}
Let~$\cT = \langle \cA, \cB \rangle$ and assume~$\cT$ is hfd-closed.
If the component~$\cA \subset \cT$ is admissible, then~$\cT^\lhf = \langle \cA \cap \cT^\lhf, \cB \cap \cT^\lhf \rangle$,
and if~$\cB \subset \cT$ is admissible, then~$\cT^\rhf = \langle \cA \cap \cT^\rhf, \cB \cap \cT^\rhf \rangle$.
\end{corollary}

\begin{proof}
If~$\cA$ is admissible, hence in particular right admissible, 
Proposition~\ref{prop:bijectionsubcat-hfd-closed}\ref{it:reflexive-hfd-closed-bijection} implies 
that the subcategory~$\cA \cap \cT^\lhf \subset \cT^\lhf$ is left admissible.
By Lemma~\ref{lem:sod-rhfd-lhf} its orthogonal consists of objects of~$\cT^\lhf$
which vanish on $\cA$, hence it is equal to~$\cB \cap \cT^\lhf$.
The statement about~$\cT^\rhf$ is proved analogously.
\end{proof}

From Corollary~\ref{cor:indecomposability}
we immediately deduce: 

\begin{corollary}\label{cor:hfd-ref-indecomp}
The following conditions for a reflexive hfd-closed category~$\cT$ are equivalent:
\begin{renumerate}
\item 
$\cT$ is indecomposable;
\item 
$\cT^\rhf$ is indecomposable;
\item 
$\cT^\lhf$ is indecomposable.
\end{renumerate}
\end{corollary}

\begin{remark}
\label{rem:sing}
As in~\cite[Definition~1.7]{Orl06}, one can use the notion of homologically finite-dimensional objects 
to define the singularity category of an hfd-closed triangulated category~$\cT$ as
\begin{equation}\label{eq:sing-cat}
\cT^\sing \coloneqq \cT / \cT^\lhf.
\end{equation}
If~$\cT$ is idempotent complete, proper, and either smooth or regular, 
then~$\cT^\sing = 0$ by Corollary~\ref{cor:sp-hfd}.
Note also that if~$\cT$ is hfd-closed 
and~$\cT = \langle \cA_1, \dots, \cA_m \rangle$ is a semiorthogonal decomposition with admissible components then
\begin{equation}
\label{eq:sing-sod}
\cT^\sing = \langle \cA_1^\sing, \dots, \cA_m^\sing \rangle.
\end{equation}
Indeed, 
Proposition~\ref{prop:bijectionsubcat-hfd-closed}\ref{it:reflexive-hfd-closed-bijection},\ref{it:reflexive-hfd-closed-ca-hf}
and 
Corollary~\ref{cor:reflexive-hfd-closed-ct-hf} applied repeatedly
prove that all components~$\cA_i$ are hfd-closed, \mbox{$\cA_i^\lhf = \cA_i \cap \cT^\lhf$}, 
and there is a semiorthogonal decomposition~$\cT^\lhf = \langle \cA_1^\lhf, \dots, \cA_m^\lhf \rangle$,
so that the argument of~\cite[Proposition~1.10]{Orl06} implies~\eqref{eq:sing-sod}.
\end{remark}

We could also consider the category~$\cT/\cT^\rhf$ instead of~~$\cT/\cT^\lhf$,
however using the obvious equality~$(\cT^\opp)^\lhf = (\cT^\rhf)^\opp$ of subcategories in~$\cT^\opp$
we see that~$\cT/\cT^\rhf \simeq ((\cT^\opp)^\sing)^\opp$,
so this replacement would not change much.
In fact, our choice is justified by better compatibility with the geometric case, 
see Proposition~\ref{prop:hfd-geometric} and Example~\ref{ex:sing}.

\subsection{Gorenstein categories}
\label{ss:gorenstein}

Recall the subcategories~$\cT^\rhf$ and~$\cT^\lhf$ of~$\cT$ defined by~\eqref{eq:ct-rhf-lhf}.

\begin{definition}
\label{def:gorenstein}
We say that a dg-enhanced triangulated category~$\cT$ is {\sf Gorenstein}, 
if it is hfd-closed and
\begin{equation*}
\cT^\rhf = \cT^\lhf
\end{equation*}
as subcategories of~$\cT$.
In other words, $\cT$ is Gorenstein if and only if~$\cT$ is hfd-closed and for~$t \in \cT$ we have 
\begin{equation*}
\RHom_\cT(t',t) \in \Db(\kk)
\quad\text{for any~$t' \in \cT$}
\iff 
\RHom_\cT(t,t'') \in \Db(\kk)
\quad\text{for any~$t'' \in \cT$}.
\end{equation*}
If~$\cT$ is Gorenstein, we set~$\cT^\hf \coloneqq \cT^\rhf = \cT^\lhf$.
\end{definition}

The next corollary follows immediately from Definition~\ref{def:hfd-closed} and Lemma~\ref{lemma:sp-hfd}\ref{item:hfd-proper}.

\begin{corollary}
\label{cor:closed-proper}
If~$\cT$ is hfd-closed and proper then it is Gorenstein with~$\cT^\hf = \cT$.
\end{corollary}

The following result proves that Gorenstein categories always admit a Serre functor.

\begin{proposition}
\label{prop:serre-functor}
If~$\cT$ is Gorenstein, then~$\cT^\homfin$ 
admits a canonical autoequivalence~$\bS_{\cT^\hf} \colon \cT^\homfin \xrightiso{} \cT^\homfin$
with a functorial isomorphism
\begin{equation}
\label{eq:serre}
\Hom_\cT(t, \bS_{\cT^\hf}(t')) \cong \Hom_\cT(t', t)^\vee
\end{equation}
for~$t \in \cT$, $t' \in \cT^\hf$;
in particular, $\bS_{\cT^\hf}$ is a Serre functor on~$\cT^\homfin$.

Moreover, if~$\cT$ is reflexive, the Serre functor~$\bS_{\cT^\hf}$ of~$\cT^\hf$ 
extends \textup(in the sense of Lemma~\textup{\ref{lem:equivalence-extension}}\textup)
to a unique autoequivalence~$\bS_\cT$ of~$\cT$ and
\begin{equation}
\label{eq:serre-weak}
\Hom_\cT(t, \bS_{\cT}(t')) \cong \Hom_\cT(t', t)^\vee
\end{equation}
whenever~$t$ or~$t'$ is in~$\cT^\hf$.
\end{proposition}

The autoequivalence~$\bS_\cT$ is \emph{not} a Serre functor for~$\cT$ unless~$\cT$ is proper.

\begin{proof}
For any~$t' \in \cT^\hf$ the dg-module~$\RHom_\cT(t', -)^\vee$ over~$\cT$ is homologically finite-dimensional, 
hence it is representable by an object in~$\cT^\hf$. 
We denote the representing object by~$\bS_{\cT^\hf}(t')$; the standard argument 
(see~\cite[Proposition~3.4]{BK})
then proves functoriality of~$\bS_{\cT^\hf}$.
Similarly, the dg-module~$\RHom_\cT(-,t')^\vee$ is homologically finite-dimensional, 
hence it is corepresentable by an object in~$\cT^\hf$ which we denote by~$\bS_{\cT^\hf}^{-1}(t')$
and again, $\bS_{\cT^\hf}^{-1}$ 
is a functor.
It also follows that the functors~$\bS_{\cT^\hf}$ and~$\bS_{\cT^\hf}^{-1}$ are mutually inverse, 
and~$\bS_{\cT^\hf}$ satisfies the Serre duality property~\eqref{eq:serre};
in particular it is a Serre functor for~$\cT^\hf$.

Moreover, applying Lemma~\ref{lem:equivalence-extension} we obtain the unique extension of~$\bS_{\cT^\hf}$ 
to an autoequivalence~$\bS_\cT$ of~$\cT$;
the isomorphism~\eqref{eq:serre-weak} then follows from~\eqref{eq:serre}.
\end{proof}

Under the Gorenstein and reflexivity conditions, the results of Lemma~\ref{lem:sod-rhfd-lhf}, 
Theorem~\ref{thm:lhfd-right-bijection}, and Proposition~\ref{prop:bijectionsubcat-hfd-closed}
can be strengthened.
Recall that~$\Adm(\cT)$ denotes the set of all admissible subcategories in~$\cT$.

\begin{proposition}
\label{prop:Gorenstein-refl-admissible}
Let~$\cT$ be a Gorenstein 
category.
\begin{renumerate}
\item 
\label{it:gor-adm}
If~$\cA \subset \cT$ is admissible, it is Gorenstein.
\item 
\label{it:gor-bijections}
If~$\cT$ is reflexive the operation~$\cA \mapsto \cA \cap \cT^\hf$ defines bijections
\begin{equation*}
\LAdm(\cT) \xrightiso{} \RAdm(\cT^\hf),
\quad 
\RAdm(\cT) \xrightiso{} \LAdm(\cT^\hf),
\quad\text{and}\quad 
\Adm(\cT) \xrightiso{} \Adm(\cT^\hf).
\end{equation*}

\item 
\label{item:admissibility-gorenstein}
If~$\cT$ is reflexive and~$\cA \subset \cT$ is Gorenstein and left or right admissible, it is admissible.
\end{renumerate}
\end{proposition}

\begin{proof}
\ref{it:gor-adm}
Assume~$\cA \subset \cT$ is admissible. 
Then by~Proposition~\ref{prop:bijectionsubcat-hfd-closed}\ref{it:reflexive-hfd-closed-ca-hf} it is hfd-closed and
by~Proposition~\ref{prop:bijectionsubcat-hfd-closed}\ref{it:reflexive-hfd-closed-bijection} 
we have~$\cA^\rhf = \cA \cap \cT^\hf = \cA^\lhf$,
which means that~$\cA$ is Gorenstein.

\ref{it:gor-bijections}
Since~$\cT^\lhf = \cT^\hf = \cT^\rhf$, 
the first two bijections follow from Proposition~\ref{prop:bijectionsubcat-hfd-closed}\ref{it:reflexive-hfd-closed-bijection}
and Theorem~\ref{thm:lhfd-right-bijection}.
It also follows that the map~$\cA \mapsto \cA \cap \cT^\hf$ 
takes an admissible subcategory in~$\cT$ to an admissible subcategory in~$\cT^\hf$.
It remains to check that the inverse maps defined on~$\LAdm(\cT^\hf)$ and~$\RAdm(\cT^\hf)$ 
by~\eqref{eq:admissiblity-bijection-inverse} coincide on~$\Adm(\cT^\hf)$.

So, let~$\cA_0 \subset \cT^\hf$ be an admissible subcategory, so that we have semiorthogonal decompositions
\begin{equation*}
\cT^\hf = \langle \cA_0^{\perp_{\cT^\hf}}, \cA_0 \rangle,
\qquad 
\cT^\hf = \langle \cA_0, {}^{\perp_{\cT^\hf}}\cA_0 \rangle.
\end{equation*}
Then~\eqref{eq:serre} with~$t,t' \in \cT^\hf$ implies~$\cA_0^{\perp_{\cT^\hf}} = \bS_{\cT^\hf}({}^{\perp_{\cT^\hf}}\cA_0)$.
On the other hand, using~\eqref{eq:serre} again, this time with~$t \in \cT$ and~$t' \in \cT^\hf$, we deduce
\begin{equation*}
{}^{\perp_\cT}(\cA_0^{\perp_{\cT^\hf}}) =
{}^{\perp_\cT}(\bS_{\cT^\hf}({}^{\perp_{\cT^\hf}}\cA_0)) =
({}^{\perp_{\cT^\hf}}\cA_0)^{\perp_\cT}.
\end{equation*}
Comparing with~\eqref{eq:admissiblity-bijection-inverse} we see that the maps 
inverse to~\eqref{eq:admissiblity-bijection-23}
indeed agree under our assumptions.

\ref{item:admissibility-gorenstein}
If~$\cA$ is left admissible, the category~$\cA^\hf = \cA^\rhf = \cA \cap \cT^\rhf$
is right admissible in~$\cT^\hf$ by~Proposition~\ref{prop:bijectionsubcat-hfd-closed}\ref{it:reflexive-hfd-closed-bijection}.
Since both~$\cT^\hf$ and~$\cA^\hf$ have Serre functors by Proposition~\ref{prop:serre-functor},
we conclude from~\cite[Proposition~3.9]{BK} that~$\cA^\hf$ is admissible in~$\cT^\hf$.
Then part~\ref{it:gor-bijections}
implies that~$\cA \subset \cT$ is admissible; 
indeed, part~\ref{it:gor-bijections} claims that there is an admissible subcategory~$\cA_0 \subset \cT$ 
such that~$\cA_0 \cap \cT^\hf = \cA^\hf$,
but then~$\cA_0$ and~$\cA$ are left admissible and~$\cA_0 \cap \cT^\hf = \cA \cap \cT^\hf$, hence~$\cA_0 = \cA$,
so that~$\cA$ is admissible.
The same argument works in the case where~$\cA$ is right admissible.
\end{proof}

The following result generalizes~\cite[Lemma~2.15]{Kalck-Pavic-Shinder}.

\begin{proposition}
\label{prop:AB-adm}
If~$\cT$ is Gorenstein and reflexive
with~$\cT = \langle \cA, \cB \rangle$
and one component is smooth and proper or regular and proper, 
then both components are admissible and Gorenstein.
\end{proposition}
\begin{proof}
Indeed, assume that~$\cB$ is smooth and proper or regular and proper.
Then it is Gorenstein by Corollary~\ref{cor:sp-hfd},
hence admissible in~$\cT$ by Proposition~\ref{prop:Gorenstein-refl-admissible}\ref{item:admissibility-gorenstein}.
Let us prove that~$\cA$ is also Gorenstein and admissible in~$\cT$.
We have~$\cB \subset \cT^\hf$ because~$\Hom_\cT(\cB,\cA) = 0$ and~$\cB = \cB^\hf$
(again by Corollary~\ref{cor:sp-hfd}).
Therefore, we can apply the Serre functor of~$\cT^\hf$ to objects of~$\cB$.
Let us check that there is a semiorthogonal decomposition
\begin{equation*}
\cT = \langle \bS_{\cT^\hf}(\cB), \cA \rangle.
\end{equation*}
Indeed, semiorthogonality follows from Serre duality~\eqref{eq:serre}. 
On the other hand, combining Serre duality in~$\cT$ and~$\cB = \cB^\hf$ 
with full faithfulness of the embedding~$\cB \hookrightarrow \cT$, 
we obtain isomorphisms
\begin{equation}
\label{eq:hom-bp-ssb}
\Hom_{\cT}(b',\bS_{\cT^\hf}(\bS_\cB^{-1}(b))) \cong 
\Hom_\cT(\bS_\cB^{-1}(b), b')^\vee \cong 
\Hom_\cB(\bS_\cB^{-1}(b), b')^\vee \cong 
\Hom_\cB(b',b) \cong 
\Hom_\cT(b',b)
\end{equation}
for any~$b,b' \in \cB$. 
Taking~$b' = b$ we obtain a canonical morphism~$b \to \bS_{\cT^\hf}(\bS_\cB^{-1}(b))$ in~$\cT$.
Now consider the distinguished triangle 
\begin{equation*}
t \to b \to \bS_{\cT^\hf}(\bS_\cB^{-1}(b))
\end{equation*}
extending it.
Using~\eqref{eq:hom-bp-ssb} we conclude that~$\Hom(b',t) = 0$ for any~$b' \in \cB$,
hence~$t \in \cB^\perp = \cA$,
and the triangle implies that~$\cB \subset \langle \bS_{\cT^\hf}(\cB), \cA \rangle$, 
and hence~$\cT = \langle \bS_{\cT^\hf}(\cB), \cA \rangle$.
Now we see that~$\cA$ is admissible, 
hence Gorenstein by Proposition~\ref{prop:Gorenstein-refl-admissible}\ref{it:gor-adm}.

The case where~$\cA$ is smooth and proper is analogous.
\end{proof}

In the geometric situation the next corollary has been obtained earlier by Kalck and Pavic,
using the category of homologically finite objects defined in~\cite{Orl06}.

\begin{corollary}
If~$\cT$ is a Gorenstein and reflexive category with a semiorthogonal decomposition
\begin{equation*}
\cT = \langle \cE_1, \dots, \cE_m \rangle, 
\end{equation*}
such that~$\cE_i \simeq \Db(\kk)$
then~$\cT = \cT^{\homfin}$
and all subcategories~$\cE_i \subset \cT$
are admissible.
\end{corollary}
\begin{proof}
Take~$\cA = \langle \cE_1, \dots, \cE_{m-1} \rangle$ and~$\cB = \cE_m$;
then~$\cB {} \simeq \Db(\kk)$ is smooth and proper, 
hence~$\cA$ and~$\cB$ are admissible and~$\cA$ is Gorenstein by Proposition~\ref{prop:AB-adm}
and reflexive by Theorem~\ref{thm:lhfd-right-bijection}.
Iterating this argument, we conclude that all subcategories~$\cE_i \subset \cT$ are admissible.
Moreover, $\cE_i \simeq \Db(\kk) = (\Db(\kk))^\hf$,
and we conclude from Proposition~\ref{prop:bijectionsubcat-hfd-closed}\ref{it:reflexive-hfd-closed-bijection}
that~$\cE_i \cap \cT^\hf = \cE_i$, hence~$\cE_i \subset \cT^\hf$ for all~$i$,
hence~$\cT = \cT^\hf$.
\end{proof}


\section{Categorical contractions and crepancy}
\label{sec:ccc}

In this section we introduce categorical contractions,
relate homologically finite-dimensional objects in the source and target of a categorical contraction,
and introduce the notion of crepancy.

\subsection{Categorical contractions}
\label{ss:cc}

Various notions of categorical contractions originate in~\cite{Ef20};
the strongest of them is the classical notion of \emph{Verdier localization}.
In this subsection we consider two weaker conditions (see Definition~\ref{def:cc} and Definition~\ref{def:he} below).

Recall that a triangulated subcategory~$\cT' \subset \cT$ is {\sf dense} if any object of~$\cT$ 
is a direct summand of an object of~$\cT'$.
Below we use notation~$\pi_*$ for a functor~\mbox{$\tcT \to \cT$} to emphasize analogy 
with the pushforward functor~$\pi_* \colon \Db(\tX) \to \Db(X)$ for a morphism~$\pi \colon \tX \to X$ of schemes.

\begin{definition}[{\cite[Definition~1.10]{KSabs}}]
\label{def:cc}
A functor~$\pi_* \colon \tcT \to \cT$ is a {\sf categorical contraction} if it is a localization up to direct summands,
i.e., if~$\pi_* \colon \tcT \to \Ima(\pi_*)$ is a Verdier localization and~$\Ima(\pi_*) \subset \cT$ 
is a dense triangulated subcategory.
\end{definition}

\begin{remark}
\label{rem:disclaimer}
In~\cite[Definition~3.7]{Ef20} the same notion is called a localization.
We find this confusing and to avoid possible misunderstanding change the terminology.
\end{remark}

\begin{definition}[{\cite[Definition~3.2]{Ef20}}]
\label{def:he}
A functor~$\pi_* \colon \tcT \to \cT$ is a {\sf homological epimorphism}
if the extension of scalars functor~$\Ind(\pi_*)\colon \Du(\tcT) \to \Du(\cT)$ is a Verdier localization.
\end{definition}

Note that a functor~$\pi_* \colon \tcT \to \cT$ is a Verdier localization or a categorical contraction or a homological epimorphism
if and only if its opposite functor~$\pi_*^\opp \colon \tcT^\opp \to \cT^\opp$ is
(for Verdier localizations and categorical contractions this is obvious 
and for homological epimorphisms this follows from~\cite[Proposition~3.4]{Ef20}).
The relationship between these three notions is the following.

\begin{proposition} 
\label{prop:contractions}
Let $\pi_*\colon \tcT \to \cT$ be a dg-enhanced functor.
\begin{renumerate}
\item
\label{item:localization-cc}
The functor~$\pi_*$ is a Verdier localization if and only if it is a categorical contraction
and the induced map on the Grothendieck groups~$\rK_0(\tcT) \to \rK_0(\cT)$ is surjective.
\item
\label{item:cc-he}
The functor~$\pi_*$ is a categorical contraction if and only if it is a homological epimorphism
and the subcategory~$\Ker(\Ind(\pi_*)) \subset \Du(\tcT)$ is generated by~$\Ker(\pi_*) \subset \tcT$ as localizing subcategory, 
i.e., it is the smallest triangulated subcategory of~$\Du(\tcT)$ containing~$\Ker(\pi_*)$ and closed under direct sums.
\end{renumerate}
\end{proposition}
\begin{proof}
\ref{item:localization-cc}
If~$\pi_*$ is a categorical contraction and~$\rK_0(\tcT) \to \rK_0(\cT)$ is surjective,
then~\mbox{$\rK_0(\Ima(\pi_*)) = \rK_0(\cT)$} and by Thomason's theorem 
on classification of dense triangulated subcategories~\cite[Theorem~2.1]{Thomason},
we have the equivalence~$\tcT / \Ker(\pi_*) \simeq \Ima(\pi_*) = \cT$.
The other implication is obvious.

\ref{item:cc-he}
This is~\cite[Corollary~3.8]{Ef20}, which goes back to~\cite{Neeman1}
(recall the difference in terminology emphasized in Remark~\ref{rem:disclaimer}).
\end{proof}

The right and left adjoins~$(\pi_*)^!$ and~$(\pi_*)^*$ of a functor~$\pi_* \colon \tcT \to \cT$
on appropriate categories of homologically finite-dimensional objects 
have been constructed in Corollary~\ref{cor:shriek-adjoint} and Lemma~\ref{lem:star-adjunction}.
To simplify the notation and to keep the geometric analogy, 
we write~$\pi^!$ instead of~$(\pi_*)^!$ for the right adjoint functor~$\Dfd{\cT} \to \Dfd{\tcT}$ 
and similarly we write~$\pi^*$ instead of~$(\pi_*)^*$.
If~$\tcT$ and~$\cT$ are hfd-closed, using the simplifications of Proposition~\ref{prop:adjunctions}
and the above conventions we can rewrite the adjunctions as
\begin{equation}
\label{eq:adj}
\begin{aligned}
\Hom_{\cT}(\pi_*\cF, \cG) &\cong \Hom_{\tcT}(\cF, \pi^!\cG),
&\qquad& 
\text{for any~$\cF \in \tcT$ and~$\cG \in \cT^\rhf$, and}
\\
\Hom_{\cT}(\cG, \pi_*\cF) &\cong \Hom_{\tcT}(\pi^*\cG, \cF).
&& 
\text{for any~$\cF \in \tcT$ and~$\cG \in \cT^\lhf$.}
\end{aligned}
\end{equation}

Imposing further assumptions on~$\pi_*$, we obtain the following result.

\begin{proposition}
\label{prop:hfd-images}
Let~$\pi_* \colon \tcT \to \cT$ be a homological epimorphism.

\begin{enumerate}[label={\textup{(\roman*)}}]
\item
\label{item:ind-pi-shriek-star}
The composition~$\Ind(\pi_*) \circ \pi^! \colon \Dfd{\cT} \to \bD(\cT)$
is isomorphic to the natural embedding.
\item
\label{item:pi-star-pi-shriek-star}
If~$\tcT$ is hfd-closed, then~$\cT$ is also hfd-closed, and the compositions
\begin{equation*}
\pi_* \circ \pi^! \colon \cT^\rhf \to \cT
\qquad\text{and}\qquad
\pi_* \circ \pi^* \colon \cT^\lhf \to \cT
\end{equation*}
are isomorphic to the natural embeddings;
in particular, $\pi^!$ and~$\pi^*$ are fully faithful.
\item
\label{item:hfd-images}
If~$\tcT$ is hfd-closed and~$\pi_*$ is a categorical contraction then the functors
\begin{equation}
\label{eq:hfd-images}
\pi^! \colon \cT^\rhf \to \tcT^\rhf \cap \Ker(\pi_*)^\perp,
\qquad\text{and}\qquad
\pi^* \colon \cT^\lhf \to \tcT^\lhf \cap {}^\perp \Ker(\pi_*)
\end{equation}
are equivalences of categories with inverse functors given by the restrictions of~$\pi_*$.
\end{enumerate}
\end{proposition}

\begin{proof}
\ref{item:ind-pi-shriek-star}
If~$\pi_* \colon \tcT \to \cT$ is a homological epimorphism the functor~$\Res(\pi_*) \colon \Du(\cT) \to \Du(\tcT)$
is fully faithful by~\cite[Proposition~3.4]{Ef20}, and since~$\Ind(\pi_*)$ is left adjoint to~$\Res(\pi_*)$, we obtain
\begin{equation}
\label{eq:ind-res}
\Ind(\pi_*) \circ \Res(\pi_*) \cong \id_{\bD(\cT)}.
\end{equation}
Restricting this isomorphism to~$\Dfd{\cT} \subset \bD(\cT)$ and using~\eqref{eq:shriek}
we obtain the claim.

\ref{item:pi-star-pi-shriek-star}
Assume~$\tcT$ is hfd-closed, hence $\Dfd{\tcT} \subset \tcT$.
It follows from~\ref{item:ind-pi-shriek-star}, Lemma~\ref{lem:res-shriek}, and~\eqref{eq:ind-pis} that
\begin{equation*}
\Dfd{\cT} = 
\Ind(\pi_*)(\Res(\pi_*)(\Dfd{\cT})) \subset 
\Ind(\pi_*)(\Dfd{\tcT}) \subset
\Ind(\pi_*)(\tcT) =
\pi_*(\tcT) 
\subset \cT.
\end{equation*}
A similar computation with the opposite categories
proves that~$\cT$ is hfd-closed.
It also follows that the composition~$\pi_* \circ \pi^!$ on~$\cT^\rhf$ 
is isomorphic to the composition~$\Ind(\pi_*) \circ \pi^!$, i.e., to the natural embedding~$\cT^\rhf \hookrightarrow \cT$,
and the same argument works for~$\pi_* \circ \pi^*$ after passing to the opposite categories.
Finally, the adjunctions~\eqref{eq:adj} between~$\pi^*$, $\pi_*$, and~$\pi^!$ 
imply the full faithfulness of~$\pi^!$ and~$\pi^*$.

\ref{item:hfd-images}
The functors are fully faithful by part~\ref{item:pi-star-pi-shriek-star}
and the fact that their images are contained in the right-hand sides
follows from the adjunctions~\eqref{eq:adj},
so we only need to check essential surjectivity.
The adjunction between~$\Ind(\pi_*)$ and~$\Res(\pi_*)$
combined with~\eqref{eq:ind-res} implies that we have a semiorthogonal decomposition
\begin{equation*}
\bD(\tcT) = \langle \Res(\pi_*)(\bD(\cT)), \Ker(\Ind(\pi_*)) \rangle.
\end{equation*}
Moreover, if~$\pi_*$ is a categorical contraction 
the category~$\Ker(\Ind(\pi_*))$ is generated by~$\Ker(\pi_*)$ by Proposition~\ref{prop:contractions}\ref{item:cc-he}, hence
\begin{equation*}
\Res(\pi_*)(\bD(\cT))
= \Big( \Ker(\Ind(\pi_*)) \Big)^\perp
= \Ker(\pi_*)^\perp,
\end{equation*}
and therefore for any object~$\tM \in \tcT^\rhf \cap \Ker(\pi_*)^\perp$
there is~$M \in \bD(\cT)$ such that~$\tM \cong \Res(\pi_*)(M)$
and we only need to show that~$M$ is right homologically finite-dimensional.
Indeed, 
for any~$\tilde{t} \in \tcT$ the complex
\begin{equation*}
M(\pi_*(\tilde{t})) \cong \Res(\pi_*)(M)(\tilde{t}) \cong \tM(\tilde{t})
\end{equation*}
is finite-dimensional because~$\tM \in \tcT^\rhf$, 
and since every~$t \in \cT$ is a direct summand of~$\pi_*(\tilde{t})$ for some~$\tilde{t} \in \tcT$,
we see that~$M(t)$ is also finite-dimensional.
This shows that~$\tM \in \pi^!(\cT^\rhf)$ and proves the first equality in~\eqref{eq:hfd-images}.
The second equality follows from the first by passing to the opposite categories.
\end{proof}

\subsection{Gorenstein property and crepancy}
\label{ss:ccc}

In this subsection we introduce the notion of crepancy in the context of categorical contractions;
its relation to other definitions
will be explained in Corollary~\ref{cor:crepant-contractions-vs-resolutions}.

\begin{definition}
\label{def:crepancy}
Let~$\pi_* \colon \tcT \to \cT$ be a categorical contraction.
We say that~$\pi_*$ is {\sf crepant} if~$\tcT$ and~$\cT$ are Gorenstein 
and the adjoint functors~$\pi^! \colon \cT^\hf \to \tcT^\hf$ and~$\pi^* \colon \cT^\hf \to \tcT^\hf$
are isomorphic.
\end{definition}

The following two lemmas provide useful criteria for crepancy.

\begin{lemma}
\label{lem:crepancy-criterion}
Let~$\pi_* \colon \tcT \to \cT$ be a categorical contraction. The following conditions are equivalent:
\begin{aenumerate}
\item 
\label{it:gorenstein1}
$\tcT$ and~$\cT$ are Gorenstein and~$\pi_*$ is crepant;
\item 
\label{it:gorenstein2}
$\tcT$ is Gorenstein and~$\tcT^\hf \cap {}^\perp \Ker(\pi_*) = \tcT^\hf \cap \Ker(\pi_*)^\perp$.
\end{aenumerate}
\end{lemma}
\begin{proof}
\ref{it:gorenstein1} $\implies$ \ref{it:gorenstein2}
By Proposition~\ref{prop:hfd-images}\ref{item:hfd-images}
the intersections
$\tcT^\hf \cap {}^\perp \Ker(\pi_*)$ and
$\tcT^\hf \cap \Ker(\pi_*)^\perp$
are the essential images of $\pi^*$
and $\pi^!$ respectively, hence they coincide 
by definition of a crepant contraction.

\ref{it:gorenstein2} $\implies$ \ref{it:gorenstein1}
The category~$\tcT$ is hfd-closed because it is Gorenstein, 
hence~$\cT$ is hfd-closed by Proposition~\ref{prop:hfd-images}\ref{item:pi-star-pi-shriek-star}.
Moreover, using Proposition~\ref{prop:hfd-images}\ref{item:hfd-images} and the assumption of~\ref{it:gorenstein2},
we obtain 
\begin{equation*}
\cT^\rhf = \pi_*(\tcT^\hf \cap \Ker(\pi_*)^\perp) = \pi_*(\tcT^\hf \cap {}^\perp\Ker(\pi_*)) = \cT^\lhf,
\end{equation*}
hence~$\cT$ is Gorenstein.
Finally, $\pi^!$ and~$\pi^*$ are isomorphic as they are both isomorphic to the inverse functor
of~$\pi_* \colon \tcT^\hf \cap \Ker(\pi_*)^\perp = \tcT^\hf \cap {}^\perp\Ker(\pi_*) \to \cT^\hf$.
\end{proof}

Recall from Corollary~\ref{cor:closed-proper} that a proper hfd-closed category~$\cT$ is Gorenstein with~$\cT^\hf = \cT$,
and hence it has a Serre functor by Proposition~\ref{prop:serre-functor}.

\begin{lemma}
\label{lemma:Serre-kernel}
Assume\/~$\tcT$ is proper and hfd-closed with a Serre functor~$\bS_\tcT$.
Let~$\pi_* \colon \tcT \to \cT$ be a categorical contraction.
If the subcategory~$\Ker(\pi_*) \subset \tcT$ is {\sf Serre-invariant}, i.e.
\begin{equation*}
\bS_\tcT(\Ker(\pi_*)) = \Ker(\pi_*),
\end{equation*}
then~$\cT$ is Gorenstein and~$\pi_*$ is crepant.
\end{lemma}

\begin{proof}
As we already noticed, $\tcT$ is Gorenstein by Corollary~\ref{cor:closed-proper}.
Moreover, the subcategories~${}^\perp \Ker(\pi_*)$ and~$\Ker(\pi_*)^\perp$ in~$\tcT$ coincide by Serre duality, 
hence the result follows from Lemma~\ref{lem:crepancy-criterion}.
\end{proof}

Finally, we relate crepant categorical contractions to categorical resolutions defined in~\cite{K08}.

\begin{corollary}
\label{cor:crepant-contractions-vs-resolutions}
If\/~$\tcT$ is idempotent complete, smooth and proper and~$\pi_* \colon \tcT \to \cT$ is a categorical contraction then~$(\tcT,\pi^*,\pi_*)$
provides a categorical resolution for~$\cT$.
If, moreover, $\pi_*$ is a crepant contraction, the corresponding resolution is weakly crepant.
\end{corollary}

\begin{proof}
The functors~$\pi^* \colon \cT^\hf \to \tcT$ and~$\pi_* \colon \tcT \to \cT$ are adjoint by~\eqref{eq:adj}, 
and the composition~$\pi_* \circ \pi^*$ is isomorphic to the natural embedding
by Proposition~\ref{prop:hfd-images}\ref{item:pi-star-pi-shriek-star},
hence~$(\tcT,\pi^*,\pi_*)$ is a categorical resolution for~$\cT$.
If, moreover, $\pi_*$ is crepant, then~$\pi^*$ is biadjoint to~$\pi_*$ by definition of crepancy,
hence the resolution is weakly crepant (see~\cite[Definition~3.4]{K08}).
\end{proof}


\section{Examples}
\label{sec:examples}

In this section we illustrate the techniques developed in the previous sections
on examples of geometric and algebraic origin.

\subsection{Projective schemes}
\label{sec:geometry}

Let~$X$ be a projective scheme over a field~$\kk$.
We use the following notation:
\begin{itemize}
\item 
$\Dqc(X)$ is the unbounded derived category of quasicoherent sheaves,
\item 
$\Db(X) \subset \Dqc(X)$ is the bounded derived category of coherent sheaves, and
\item 
$\Dp(X) \subset \Db(X)$ is the subcategory of perfect complexes.
\end{itemize}
The categories~$\Db(X)$ and~$\Dp(X)$ are essentially small, idempotent complete, and have natural dg-enhancements.
Moreover, the objects of~$\Dp(X)$ form a set of compact generators of~$\Dqc(X)$, hence
\begin{equation}
\label{eq:bd-dp-dqc}
\Dqc(X) \simeq \bD(\Dp(X))
\end{equation} 
by Lemma~\ref{lem:dqc-dp}; we use this identification in Proposition~\ref{prop:hfd-geometric}.

The categories~$\Dp(X)$ and~$\Db(X)$ are self-dual by means of the naive duality
\begin{align*}
\Dp(X)^\opp &\xrightiso{} \Dp(X),
& 
\cF \mapsto \cRHom(\cF,\cO_X),
\\
\intertext{and the Grothendieck duality}
\Db(X)^\opp &\xrightiso{} \Db(X),
& 
\cG \mapsto \cRHom(\cG,\omega_X^\bullet),
\end{align*}
where~$\omega_X^\bullet$ is the dualizing complex.
The composition of these dualities gives an equivalence of categories~$\Dp(X) \xrightiso{} \Dp(X) \otimes \omega_X^\bullet$
defined by~$\cF \mapsto \cF \otimes \omega_X^\bullet$.

The following result is a restatement of some well-known results from the literature, 
using the language of homologically finite-dimensional objects.
For a much more advanced version of the duality between the $\infty$-categories of coherent sheaves and perfect complexes
on perfect derived stacks we refer to~\cite[Theorem~1.1.3, Theorem~1.2.4, Remark~1.2.6]{BZNP}.

\begin{proposition}
\label{prop:hfd-geometric}
Let~$X$ be a projective scheme 
over a perfect field~$\kk$ with dualizing complex~$\omega_X^\bullet$.
\begin{renumerate}
\item
\label{item:dpx-hf}
The category~$\Dp(X)$ is proper and
\begin{equation}\label{eq:DpX-lhfd-rhf}
\Dfd{\Dp(X)} = \Db(X),
\end{equation}
as subcategories of~$\bD(\Dp(X)) = \Dqc(X)$.

\item
\label{item:dbx-hf}
The category~$\Db(X)$ is smooth and hfd-closed with
\begin{equation}
\label{eq:DbX-lhfd-rhf}
(\Db(X))^\lhf = \Dp(X)
\qquad\text{and}\qquad
(\Db(X))^\rhf = \Dp(X) \otimes \omega_X^\bullet,
\end{equation}
as subcategories of~$\Db(X)$.

\item
\label{item:dbx-dpx-reflexivity}
Both~$\Dp(X)$ and~$\Db(X)$ are reflexive. 

\item
\label{item:dbx-gorenstein}
The category~$\Db(X)$ is Gorenstein if and only if~$X$ is Gorenstein.
\end{renumerate}
\end{proposition}

\begin{proof}
\ref{item:dpx-hf}
Properness of~$\Dp(X)$ is obvious, and an identification~$\Dfd{\Dp(X)} = \Db(X)$
follows from a combination of~\eqref{eq:bd-dp-dqc} with~\cite[Theorem~A.1]{Bondal-vdB} and Lemma~\ref{lem:representability}.

\ref{item:dbx-hf} 
The category~$\Db(X)$ is smooth by~\cite[Theorem~6.3]{Lunts}, 
hence by Corollary~\ref{cor:sp-hfd}
it is hfd-closed.
Moreover, $(\Db(X))^\lhf = \Dp(X)$ by~\cite[Proposition~1.11]{Orl06},
and the description of~$(\Db(X))^\rhf$ follows from this by Grothendieck duality.

\ref{item:dbx-dpx-reflexivity}
Follows from Lemma~\ref{lem:reflexivity-criterion} applied to the proper category~$\cT = \Dp(X)$
combined with~\ref{item:dpx-hf} and~\ref{item:dbx-hf}.

\ref{item:dbx-gorenstein}
By~\ref{item:dbx-hf} the Gorenstein condition for~$\Db(X)$
is equivalent to the equality~$\Dp(X) \otimes \omega_X^\bullet = \Dp(X)$, as subcategories of~$\Db(X)$.
Of course, this holds true 
if and only if the dualizing complex~$\omega_X^\bullet$ is perfect,
i.e, if and only if~$X$ is Gorenstein (see~\cite[Lemma 6.6]{Ballard} or~\cite[Lemma~6.25]{Lunts} 
or~\cite[Proof of Lemma~2.14]{Kalck-Pavic-Shinder}).
\end{proof}

Using Proposition~\ref{prop:hfd-geometric} we can interpret constructions from the previous sections geometrically.

\begin{example}
\label{ex:sing}
Recall the definition~\eqref{eq:sing-cat} of the singularity category~$\cT^\sing$.
For~$\cT = \Db(X)$ we see that
\begin{equation}
\label{eq:dbsing-dsg}
(\Db(X))^\sing = 
\Db(X) / (\Db(X))^\lhf = 
\Db(X) / \Dperf(X) = 
\Dsing(X)
\end{equation}
is the classical singularity category of~$X$.
\end{example}

\begin{example}\label{ex:geometric-adj}
Let~$\pi \colon \tX \to X$ be a morphism between projective schemes over a perfect field.
Assume that~$\pi_*\cO_{\tX} \cong \cO_X$.
Then ~$\pi_*\colon \Db(\tX) \to \Db(X)$ is a homological epimorphism by~\cite[Proposition~8.12]{Ef20},
and its two adjoints constructed in Proposition~\ref{prop:hfd-images},
after identifications from Proposition~\ref{prop:hfd-geometric} become the familiar fully faithful embeddings
\begin{equation*}
\pi^*\colon \Dp(X) \to \Dp(\tX) 
\qquad\text{and}\qquad  
\pi^!\colon \Dp(X) \otimes \omega_X^\bullet \to \Dp(\tX) \otimes \omega_{\tX}^\bullet.
\end{equation*}
In particular, if both~$\tX$ and~$X$ are Gorenstein, we have two functors
\begin{equation}\label{eq:adj-geom-gorenstein}
\pi^*, \pi^!\colon \Dp(X) \to \Dp(\tX)
\end{equation}
which differ by the relative dualizing sheaf twist.
Thus, if~$\pi_* \colon \Db(\tX) \to \Db(X)$ is a categorical contraction
(see~\cite[Theorem~8.22]{Ef20} and~\cite[Theorem~5.2]{KSabs} for a sufficient condition),
it follows that~$\pi_*$ is crepant if and only if so is~$\pi$;
moreover, Proposition~\ref{prop:hfd-images}\ref{item:hfd-images} also describes the essential images of~$\pi^!$ and~$\pi^*$ 
in terms of the orthogonals to~$\Ker(\pi_*)$.
\end{example}

\begin{remark}
Similarly to the subtlety described in Remark~\ref{rem:warning-1},
it should be taken into account that the following compositions
\begin{equation*}
\Dqc(\tX) \xrightarrow{\ \pi_*\ } \Dqc(X) \hookrightarrow \bD(\Db(X))
\qquad\text{and}\qquad 
\Dqc(\tX) \hookrightarrow \bD(\Db(\tX)) \xrightarrow{\ \Ind(\pi_*)\ } \bD(\Db(X))
\end{equation*}
(where the embedding~$\Dqc(X) \hookrightarrow \bD(\Db(X))$ is induced by~$\Dp(X) \hookrightarrow \Db(X)$ and~\eqref{eq:bd-dp-dqc},
and analogously for~$\tX$)
do not agree: for instance, if we additionally assume that~$\tX$ is smooth,
the second composition is essentially surjective by~\cite[Proposition~8.12]{Ef20}, 
while the first one is not!
\end{remark}

Combining Proposition~\ref{prop:hfd-geometric} with results of~\S\ref{sec:hfd-gor} we can also relate 
admissible subcategories and semiorthogonal decompositions of~$\Db(X)$ to those of~$\Dp(X)$.
In particular, applying Proposition~\ref{prop:bijectionsubcat-hfd-closed} 
we obtain the following corollary generalizing~\cite[Theorem~A.1]{KKS20}.
A similar result was obtained independently by Bondarko in~\cite[Theorem~3.2.7]{B22}.

\begin{corollary}
\label{ex:DbDperf-decomp}
If~$X$ is a projective scheme over a perfect field with dualizing complex~$\omega_X^\bullet$, 
the operations
\begin{align*}
\cA &\mapsto \cA \cap \Dp(X) 
&
\cA &\mapsto \cA \cap (\Dp(X) \otimes \omega_X^\bullet)
\shortintertext{induce bijections}
\RAdm(\Db(X)) & \cong \LAdm(\Dp(X))
&
\LAdm(\Db(X)) & \cong \RAdm(\Dp(X) \otimes \omega_X^\bullet).
\end{align*}
If, moreover,~$X$ is Gorenstein, each operation defines a bijection~$\Adm(\Db(X)) \cong \Adm(\Dp(X))$.
\end{corollary}

Similarly, applying Corollary~\ref{cor:indecomposability} we obtain

\begin{corollary}
\label{cor:Db-ind}
If~$X$ is a projective scheme 
over a perfect field,
the category~$\Db(X)$ is indecomposable if and only if~$\Dperf(X)$ is indecomposable.
\end{corollary}

One can further combine Corollary~\ref{cor:Db-ind} with various results in the literature
establishing the indecomposability of~$\Dp(X)$ (see~\cite{Okawa,Kawatani-Okawa,Spence,LopesMartin-deSalas})
and deduce the indecomposability of~$\Db(X)$.

\begin{corollary}[{\cite[Corollary~2.9 and Remark~2.10]{LopesMartin-deSalas}}]
\label{cor:cm-db-indecomposable}
Let~$X$ be a connected Cohen--Macaulay projective variety over a perfect field.
Assume the base locus of the dualizing sheaf~$\omega_X$ is empty or consists of a finite set of points.
Then~$\Db(X)$ is indecomposable. 
\end{corollary}

For curves we obtain a simple geometric criterion for indecomposability of~$\Db(X)$ and~$\Dperf(X)$.

\begin{corollary}
\label{cor:indecomposability-curves}
Let~$X$ be a connected nodal projective curve over a perfect field.
If $X$ has no smooth rational components, then~$\Db(X)$ is indecomposable.
\end{corollary}

\begin{proof}
Let~$Y \subset X$ be an irreducible component of~$X$.
By the adjunction formula we have~$\omega_X\vert_Y = \omega_Y(D)$,
where~$D$ is the intersection of~$Y$ with~$X \setminus Y$, see, e.g., \cite[Lemma~1.12]{Catanese}.
Since we assume~$X$ has no smooth rational components, 
$\omega_X$ has positive degree on each component, 
hence the base locus of~$\omega_X$ is empty or finite, see~\cite[Theorem~D]{Catanese}.
Thus, Corollary~\ref{cor:cm-db-indecomposable} applies to give the result.
\end{proof}

Note that if~$X$ has a \emph{rational tail}, i.e., $X = \P^1 \cup X'$, where~$\P^1$ and~$X'$ intersect transversely at one point, 
then~$\Db(X)$ has a nontrivial semiorthogonal decomposition by~\cite[Proposition~6.15]{KSabs}.
On a contrary, if~$X$ is a connected semistable curve, a minor modification of the argument of Corollary~\ref{cor:cm-db-indecomposable}
also proves indecomposability of~$\Dp(X)$ and~$\Db(X)$.


\subsection{Proper connective DG-algebras}
\label{sec:algebra}

Let~$A$ be a {\sf proper connective} dg-algebra over a field~$\kk$, i.e., a dg-algebra such that
\begin{equation}
\label{eq:dg-algebra}
\sum \dim_\kk \rH^i(A) < \infty
\qquad\text{and}\qquad 
\rH^i(A) = 0\quad\text{for~$i > 0$},
\end{equation} 
where~$\rH^i(A)$ is the cohomology of~$A$ in degree~$i$.
We use the following notation:
\begin{itemize}
\item 
$\Du(A)$ is the derived category of all right dg-modules over~$A$,
\item 
$\Db(A) \subset \Du(A)$ is its subcategory of dg-modules with finite-dimensional total cohomology, and
\item 
$\Dp(A) \subset \Db(A)$ is the subcategory of perfect dg-modules.
\end{itemize}
The category~$\bD(A)$ is cocomplete, hence idempotent complete.
The categories~$\Db(A)$ and~$\Dp(A)$ are closed in~$\bD(A)$ with respect to taking direct summands, hence also idempotent complete.
Moreover, they are essentially small and have natural dg-enhancements.
Finally, similarly to the geometric case, the objects of~$\Dp(A)$ 
form a set of compact generators of~$\bD(A)$, hence
\begin{equation}
\label{eq:bd-dp-bda}
\bD(A) \simeq \bD(\Dp(A)),
\end{equation} 
again by Lemma~\ref{lem:dqc-dp}.

There are natural equivalences
\begin{align}
\notag
\Dp(A)^\opp &\xrightiso{} \Dp(A^\opp),
& 
M &\mapsto \RHom_A(M, A),
\\
\intertext{given by the duality over~$A$, and}
\label{eq:duality-k}
\Db(A)^\opp &\xrightiso{} \Db(A^\opp),
& 
N &\mapsto N^\vee,
\end{align}
given by the duality over~$\kk$, respectively 
(so that~\eqref{eq:duality-k} can be understood as a special case of~\eqref{eq:duk}).

Below we will also use the fact that the category~$\bD(A)$ is endowed with a t-structure, where
\begin{equation}
\label{eq:tstr}
\bD(A)^{\le 0} = \{ M \mid \rH^i(M) = 0\ \text{for~$i > 0$}\},
\qquad 
\bD(A)^{\ge 0} = \{ M \mid \rH^i(M) = 0\ \text{for~$i < 0$}\},
\end{equation} 
see~\cite[Theorem~1.3]{HoshinoKatoMiyachi}.
It obviously induces a t-structure on~$\Db(A)$, defined in the same way.

Recall that~$\thick(S) \subset \cT$ denotes the thick subcategory of~$\cT$ 
generated by a set of objects~$S \subset \cT$.
Note that~$\Dp(A) = \thick(A)$, where~$A$ in the right side is understood as the free right $A$-module.

Following~\cite[Assumption~0.1(3)]{Jin} we say that a dg-algebra~$A$ is {\sf Gorenstein} if
\begin{equation*}
\thick(A^\vee) = \thick(A),
\end{equation*}
as subcategories of~$\Du(A)$ (hence both are equal to~$\Dp(A)$),
where~$A^\vee$ is the image in~$\Db(A)$ of the free left $A$-module~$A$ under the equivalence~\eqref{eq:duality-k}.

\begin{proposition}
\label{prop:hfd-algebraic}
Let~$A$ be a proper connective dg-algebra over a perfect field.
\begin{renumerate}
\item
\label{item:dpa-hf}
The category~$\Dp(A)$ is proper and
\begin{equation}
\label{eq:DpA-lhfd-rhf}
\Dfd{\Dp(A)} = \Db(A)
\end{equation}
as subcategories of~$\bD(\Dp(A)) = \bD(A)$.

\item
\label{item:dba-hf}
The category~$\Db(A)$ is smooth and hfd-closed with
\begin{equation}
\label{eq:DbA-lhfd-rhf}
(\Db(A))^\lhf = \thick(A) = \Dp(A),
\qquad\text{and}\qquad
(\Db(A))^\rhf = \thick(A^\vee),
\end{equation}
as subcategories of~$\Db(A)$.

\item
\label{item:dba-dpa-reflexivity}
Both~$\Dp(A)$ and~$\Db(A)$ are reflexive. 

\item
\label{item:dba-gorenstein}
The category~$\Db(A)$ is Gorenstein if and only if~$A$ is Gorenstein.
\end{renumerate}
\end{proposition}

\begin{proof}
\ref{item:dpa-hf}
Properness of~$\Dp(A)$ follows from properness of~$A$.
Assume~$M \in \bD(\Dp(A)) = \bD(A)$ is homologically finite-dimensional.
Then
\begin{equation*}
\rH^\bullet(M) = \rH^\bullet(\RHom_{A}(A,M))
\end{equation*}
is finite-dimensional, hence~$M \in \Db(A)$, and therefore~$\Dfd{\Dp(A)} \subset \Db(A)$.
The converse inclusion also follows from the above equality because~$\Dp(A) 
= \thick(A)$.

\ref{item:dba-hf} 
Note that~$A$ is quasiisomorphic to a finite-dimensional connective dg-algebra by~\cite[Corollary~3.12]{RS19},
so we may assume~$A$ is finite-dimensional.
Let~$\tA$ be the \emph{Auslander resolution} of~$A$ 
associated to the radical~$\rad(A)$ and an integer~$N$ such that~$\rad(A)^N = 0$
(see~\cite[\S5]{KL} or~\cite[\S2.3]{O20});
note that by construction~$\tA$ is also a finite-dimensional connective dg-algebra,
and there is a distinguished (closed) 
idempotent element~$\be \in \tA$ such that
\begin{equation*}
\be \cdot \tA \cdot \be = A.
\end{equation*}
Consequently, there is a pair of continuous adjoint functors
\begin{align*}
\pi^* &\colon \bD(A) \to \bD(\tA),
&
M &\mapsto M \otimes_A (\be \cdot \tA),
\\
\pi_* & \colon \bD(\tA) \to \bD(A),
&
\tM & \mapsto \RHom_{\tA}(\be \cdot \tA, \tM) = \tM \cdot \be
\end{align*}
such that~$\pi_* \circ \pi^* \cong \id_{\bD(A)}$.
Note the inclusion
\begin{equation}
\label{eq:pis-db-ta}
\pi_*(\Db(\tA)) \subset \Db(A)
\end{equation} 
which is obvious from the above.
Note also that the category~$\Db(\tA)$ is smooth and proper (\cite[Theorem~5.20]{KL} or~\cite[Theorem~2.19(5)]{O20})), 
hence (using part~\ref{item:dpa-hf} and Lemma~\ref{lemma:sp-hfd}) we obtain
\begin{equation*}
\Db(\tA) = \Dfd{\Dp(\tA)} = \Dp(\tA).
\end{equation*}

Now, using the fact that~$\pi_*$ is t-exact for the natural t-structures~\eqref{eq:tstr} on~$\Db(\tA)$ and~$\Db(A)$,
we deduce that~$\pi_* \colon \Db(\tA) \to \Db(A)$ is a Verdier localization,
see the argument of~\cite[Corollary~A.13]{KL} or~\cite[Lemma~2.32]{PS18}.
Moreover, using~\cite[Corollary~2.9 and Theorem~2.4(1) (see also Remark~2.7)]{Ef20} 
we see that smoothness of~$\Db(\tA)$ implies that~$\Db(A)$ is smooth.
Therefore~$\Db(A)$ is hfd-closed by Corollary~\ref{cor:sp-hfd}.

Now we show that for~$M \in \bD(A)$ the object~$\pi^*(M)$ is compact if and only if~$M$ is compact.
Indeed, the isomorphism~$\pi_* \circ \pi^* \cong \id_{\bD(A)}$ implies that~$\pi_*$ is essentially surjective,
hence for any collection of dg-modules~$N_i \in \bD(A)$ we can write~$N_i = \pi_*(\tN_i)$ for appropriate~$\tN_i \in \bD(\tA)$, 
so if~$\pi^*(M)$ is compact, using the adjunction and continuity of~$\pi_*$ we obtain
\begin{multline*}
\Hom(M, \oplus N_i) = 
\Hom(M, \oplus \pi_*(\tN_i)) \cong 
\Hom(M, \pi_*(\oplus \tN_i)) \cong 
\Hom(\pi^*(M), \oplus \tN_i) \\ \cong 
\oplus \Hom(\pi^*(M), \tN_i) \cong 
\oplus \Hom(M, \pi_*(\tN_i)) = 
\oplus \Hom(M, N_i),
\end{multline*}
hence~$M$ is compact, and a similar computation proves the converse implication.

Further, since the compact objects in~$\bD(A)$ and~$\bD(\tA)$ are perfect dg-modules,
the above observation implies that~\mbox{$\pi^*(M) \in \Dp(\tA)$} if and only if~$M \in \Dp(A)$.
On the other hand, 
we have
\begin{equation*}
\pi^*(\Db(A)^\lhf) \subset \Db(\tA)^\lhf = \Db(\tA) = \Dp(\tA),
\end{equation*}
where the inclusion follows from the adjunction of~$\pi^*$ and~$\pi_*$ and~\eqref{eq:pis-db-ta}.
Thus we have~$\Db(A)^\lhf \subset \Dp(A)$.
The opposite inclusion~$\Dp(A) \subset \Db(A)^\lhf$ follows from the argument of part~\ref{item:dpa-hf}.

Finally, the description of the subcategory~$\Db(A)^\rhf \subset \Db(A)$ now follows from the equivalence~\eqref{eq:duality-k} 
and the equality~$\Dp(A^\opp) = \thick(A)$ obtained by the above argument applied to the opposite algebra.

\ref{item:dba-dpa-reflexivity}
Follows from Lemma~\ref{lem:reflexivity-criterion}
applied to the proper category~$\cT = \Dp(A)$
combined with~\ref{item:dpa-hf} and~\ref{item:dba-hf}.

\ref{item:dba-gorenstein}
Follows immediately from Definition~\ref{def:gorenstein} and~\ref{item:dba-hf}.
\end{proof}

Here is a simple example of a proper connective dg-algebra that played the 
key in role in~\cite{KSabs}.

\begin{example}
\label{ex:Ap}
Consider the dg-algebra~$\sA_p = \kk[\eps]/(\eps^2)$, where~$\deg(\eps) = -p$ with~$p \ge 0$ and~$\rd(\eps) = 0$.
It is Gorenstein since~$\sA_p^\vee \cong \sA_p[-p]$ as~$\sA_p$-modules.
Furthermore, the category~$\Db(\sA_p)^\hf = \Dp(\sA_p)$ 
was shown in~\cite[Proposition~2.2]{KSabs} to be equivalent to the category~$\Db(\sB_{p+1})$
for the dg-algebra~$\sB_{p+1} = \kk[\uptheta]$, where~$\deg(\uptheta) = p + 1$ and~$\rd(\uptheta) = 0$.
It was also observed (\cite[Remark~6.8]{KSabs}) that
the singularity category~$\Db(\sA_p)^\sing = \Db(\sA_p) / \Dp(\sA_p)$ 
is equivalent to the category of~$\ZZ/(p+1)\ZZ$-graded vector spaces;
in particular, it is idempotent complete.
This observation was crucial in~\cite{KSabs} for establishing obstructions 
for the existence of categorical absorption of singularities.
\end{example}

The following observation combines the geometric and algebraic examples.

\begin{corollary}
\label{cor:bij-X-A}
Let~$X$ be a projective Gorenstein scheme over a perfect field.
An admissible subcategory~$\cA \subset \Db(X)$ is equivalent to~$\Db(A)$ for a proper connective dg-algebra~$A$
if and only if the category~$\cA \cap \Dp(X)$ is equivalent to~$\Dp(A)$.
\end{corollary}
\begin{proof}
By Proposition~\ref{prop:hfd-geometric} the category~$\Db(X)$ is Gorenstein and reflexive with~$(\Db(X))^\hf = \Dp(X)$.

If~$\cA \subset \Db(X)$ is admissible, it is Gorenstein by Proposition~\ref{prop:Gorenstein-refl-admissible}\ref{it:gor-adm},
so if~$\cA \simeq \Db(A)$ for a proper connective dg-algebra~$A$
then~$A$ is Gorenstein by Proposition~\ref{prop:hfd-algebraic}\ref{item:dba-gorenstein}.
Moreover,
\begin{equation*}
\cA \cap \Dp(X) = \cA \cap (\Db(X))^\hf = \cA^\hf
\end{equation*}
by Proposition~\ref{prop:bijectionsubcat-hfd-closed}\ref{it:reflexive-hfd-closed-bijection},
hence by Proposition~\ref{prop:hfd-algebraic}\ref{item:dba-hf} this is equivalent to~$\Dp(A)$.

Conversely assume that~$\cA$ is admissible and~$\cA \cap \Dp(X) \simeq \Dperf(A)$.
By Proposition~\ref{prop:bijectionsubcat-hfd-closed}\ref{it:reflexive-hfd-closed-bijection} 
the map~$\cA \mapsto \cA \cap \Dp(X)$ coincides with the map of Lemma~\ref{lem:sod-rhfd-lhf},
and by Theorem~\ref{thm:lhfd-right-bijection} the inverse map 
takes an admissible subcategory~$\cA_0 \subset \Dp(X)$ to~$\Dfd{\cA_0} \subset \Db(X)$.
Therefore,
\begin{equation*}
\cA \simeq \Dfd{\cA \cap \Dp(X)} \simeq \Dfd{\Dp(A)} \simeq \Db(A)
\end{equation*}
where the last equivalence is Proposition~\ref{prop:hfd-algebraic}\ref{item:dpa-hf}.
\end{proof}

In the remaining part of the section we generalize Example~\ref{ex:Ap} by showing that 
for any Gorenstein proper connective dg-algebra~$A$
the singularity category of~$\Db(A)$ defined by~\eqref{eq:sing-cat}, i.e.,
\begin{equation*}
\Dsing(A) := \Db(A)^\sing = \Db(A) / \Db(A)^\lhf = \Db(A) / \Dp(A)
\end{equation*}
is idempotent complete.

Recall that an additive category is called a {\sf Krull--Schmidt category}
if every object admits a finite direct sum decomposition into objects with local endomorphism rings, see~\cite[\S4]{Krause-KS}.

\begin{proposition}\label{prop:Gorenstein-Dsing}
If~$A$ is a Gorenstein proper connective dg-algebra over a perfect field
then the singularity category~\mbox{$\Dsing(A)$} is Krull--Schmidt.
Moreover, $\Dsing(A)$ is idempotent complete.
\end{proposition}
\begin{proof}
First, we prove that~$\Db(A)$ is a Krull--Schmidt category.
Indeed, as we already mentioned, the category~$\Db(A)$ is idempotent complete,
hence by~\cite[Corollary~4.4 and Proposition~4.1]{Krause-KS} 
it is enough to check that the category of finitely-generated projective~$\End(M)$-modules is Krull--Schmidt for every~$M \in \Db(A)$.
But the ring~$\End(M)$ is finite-dimensional by~\cite[Proposition~1.5]{Jin},
hence~\cite[Section~5]{Krause-KS} applies and proves the Krull--Schmidt property.

To prove the Krull--Schmidt property for~$\Dsing(A)$ we use~\cite[Theorem~0.3(4)]{Jin}, i.e., an equivalence
\begin{equation*}
\Dsing(A) \simeq \underline{\CM}(A),
\end{equation*}
where~$\underline{\CM}(A)$ is the stable category of Cohen--Macaulay dg-modules over~$A$.
Recall from~\cite{Jin} that a dg-module~$M$ over~$A$ is {\sf Cohen--Macaulay} if
\begin{equation*}
M \in \Db(A)^{\le 0} 
\qquad\text{and}\qquad  
\RHom_{A}(M, A) \in \Db(A^\opp)^{\le 0},
\end{equation*}
(where we use the t-structure~\eqref{eq:tstr})
and that the stable category~$\underline{\CM}(A)$ of Cohen--Macaulay modules 
is defined from the category~$\CM(A)$ of Cohen--Macaulay dg-modules
by quotienting out the set of morphisms in~$\CM(A)$ which factor through a finite direct sum of shifts of~$A$.

Since~$\CM(A)$ is closed under direct summands taken in the Krull--Schmidt category~$\Db(A)$, 
the category~$\CM(A)$ is Krull--Schmidt category as well.
Since quotients of local rings are local,
the quotient category~$\underline{\CM}(A)$ is also a Krull--Schmidt category,
hence so is~$\Dsing(A)$.

Finally, every Krull--Schmidt category is idempotent complete by~\cite[Corollary~4.4]{Krause-KS}.
\end{proof}


\section{Bijections between sets of admissible subcategories}
\label{sec:bijection}

The goal of this section is to prove Theorem~\ref{thm:bijection-subcat-deform} from the Introduction.

We use notation introduced in~\S\ref{sec:geometry}.
Furthermore, for a triangulated category~$\cT$, 
let~$\Sub(\cT)$ denote the set of all strict (that is, closed under isomorphism),
but not necessarily triangulated subcategories of~$\cT$.

\begin{definition}
Given two triangulated categories $\cT$, $\cT'$
we say that a map~$\Upsilon \colon \Sub(\cT) \to \Sub(\cT')$
\begin{itemize}
\item 
{\sf preserves all semiorthogonal decompositions} if~$\Upsilon(0) = 0$ 
and for any semiorthogonal decomposition~$\cT = \langle \cA, \cB \rangle$
we have a semiorthogonal decomposition~$\cT' = \langle \Upsilon(\cA), \Upsilon(\cB) \rangle$;
\item 
{\sf preserves semiorthogonal decompositions with an admissible component}
if the same holds whenever~$\cA$ or~$\cB$ is admissible.
\end{itemize}
\end{definition}

Note that if~$\Upsilon$ preserves all semiorthogonal decompositions 
then of course if preserves semiorthogonal decompositions with an admissible component.

\begin{example}
\label{ex:ups-hf}
If~$\cT$ is a Gorenstein category, Corollary~\ref{cor:reflexive-hfd-closed-ct-hf} implies that the map
\begin{equation*}
\Upsilon^\hf \colon \Sub(\cT) \to \Sub(\cT^\hf),
\qquad
\cA \mapsto \cA \cap \cT^\hf
\end{equation*}
preserves semiorthogonal decompositions with an admissible component.
\end{example}

The following property of maps preserving semiorthogonal decompositions is easy.

\begin{lemma}
\label{lem:sod-adm}
Consider a map $\Upsilon \colon \Sub(\cT) \to \Sub(\cT')$.
If~$\Upsilon$ preserves semiorthogonal decompositions with an admissible component then
\begin{equation*}
\Upsilon(\cT) = \cT'
\qquad\text{and}\qquad
\Upsilon(\Adm(\cT)) \subset \Adm(\cT').
\end{equation*}
\end{lemma}

\begin{proof}
The equality follows by applying~$\Upsilon$ to the semiorthogonal decomposition~$\cT = \langle 0, \cT \rangle$;
and the inclusion follows by applying~$\Upsilon$ 
to semiorthogonal decompositions~$\cT = \langle \cA, {}^\perp\cA \rangle$ and~$\cT = \langle \cA^\perp, \cA \rangle$.
\end{proof}

The following geometric example will be important for the proof of the theorem.

\begin{lemma}
\label{lem:ups-pf-Bijection}
If~$X$ is a projective Gorenstein scheme over a perfect field, the map
\begin{equation}
\label{eq:ups-pf-adm}
\Upsilon^\pf \colon \Sub(\Db(X)) \to \Sub(\Dp(X)),
\qquad
\cA \mapsto \cA \cap \Dp(X)
\end{equation}
preserves semiorthogonal decompositions with an admissible component
and induces a bijection
\begin{equation*}
\Adm(\Db(X)) \xrightiso{} \Adm(\Dp(X)).
\end{equation*}
\end{lemma}

\begin{proof}
By Proposition~\ref{prop:hfd-geometric} the category~$\cT = \Db(X)$ is reflexive and Gorenstein, and~$\Db(X)^\hf = \Dp(X)$.
Therefore, $\Upsilon^\pf = \Upsilon^\hf$
preserves semiorthogonal decompositions with an admissible component by Example~\ref{ex:ups-hf}
and induces a bijection on admissible subcategories by Proposition~\ref{prop:Gorenstein-refl-admissible}.
\end{proof}

The following construction gives another example of a map preserving semiorthogonal decompositions.
Recall that~$\thick(S)$ denotes the thick subcategory generated by~$S$.
For an exact functor~$\Phi\colon \cT \to \cT'$ we define the map~$\Upsilon_\Phi\colon \Sub(\cT) \to \Sub(\cT')$ by
\begin{equation}
\label{def:Gamma-Phi}
\Upsilon_\Phi(\cA) \coloneqq \thick(\Phi(\cA)).
\end{equation}
We say that a functor~$\Phi\colon \cT \to \cT'$ {\sf has dense image} if~$\cT' = \thick(\Phi(\cT))$.

\begin{lemma}
\label{lem:Gamma-adjoints}
Let~$\Phi\colon \cT \to \cT'$ be an exact functor with dense image.
If either
\begin{aenumerate}
\item 
\label{item:ra}
$\Phi$ has a right adjoint $\Phi^!$ and~$\Cone(F \to \Phi^!(\Phi(F))) \cong F[n]$
for each~$F \in \cT$ and some~$n \in \ZZ$, or
\item
\label{item:la}
$\Phi$ has a left adjoint $\Phi^*$ and~$\Cone(\Phi^*(\Phi(F)) \to F) \cong F[n]$
for each~$F \in \cT$ and some~$n \in \ZZ$,
\end{aenumerate}
then~$\Upsilon_\Phi$ preserves all semiorthogonal decompositions.
\end{lemma}

\begin{proof}
We only prove the lemma under assumption~\ref{item:ra}, the case of assumption~\ref{item:la} being analogous. 
Given~$\cT = \langle \cA, \cB \rangle$, 
for any objects~$A \in \cA$, $B \in \cB$ we have
\begin{equation*}
\Hom(\Phi(B),\Phi(A)) \cong \Hom(B,\Phi^!(\Phi(A)))
\end{equation*}
by adjunction, while~\ref{item:ra} gives a distinguished triangle
\begin{equation*}
A \to \Phi^!(\Phi(A)) \to A[n]
\end{equation*}
for some~$n \in \ZZ$.
It follows that the subcategories~$\Upsilon_\Phi(\cA) = \thick(\Phi(\cA))$ and~$\Upsilon_\Phi(\cB) = \thick(\Phi(\cB))$ 
in~$\cT'$ are semiorthogonal. 
Moreover, since~$\Phi$ has dense image, these subcategories generate~$\cT'$,
hence we have a semiorthogonal decomposition~$\cT' = \langle \Upsilon_\Phi(\cA), \Upsilon_\Phi(\cB) \rangle$.
\end{proof}

\begin{remark}
If~$\Phi$ is dg-enhanced and both~$\Phi^*$ and~$\Phi^!$ exist,
there are exact triangles of functors
\begin{equation*}
\bT_{\Phi^!,\Phi} \to \id \to \Phi^! \circ \Phi
\qquad\text{and}\qquad 
\Phi^* \circ \Phi \to \id \to \bT_{\Phi^*,\Phi},
\end{equation*}
where~$\bT_{\Phi^!,\Phi}$ and~$\bT_{\Phi^*,\Phi}$ are the so-called twist functors of~$\Phi$.
So, if~$\Phi$ is a spherical functor such that the corresponding spherical twists of~$\cT$ are shifts,
then~$\Upsilon_\Phi$ preserves semiorthogonal decompositions.
\end{remark}

Now we consider the situation of Theorem~\ref{thm:bijection-subcat-deform}.
Let~$\io \colon X \hookrightarrow \cX$ be the embedding of a projective Gorenstein variety~$X$ over a perfect field
into a smooth quasiprojective variety~$\cX$ as a Cartier divisor 
linearly equivalent to zero, i.e.,
\begin{equation}
\label{eq:normal-bundle-triv}
\cO_\cX(X) \cong \cO_\cX.    
\end{equation}
Recall that~$\Db_X(\cX)$ denotes the full subcategory of~$\Db(\cX)$ of objects 
set-theoretically supported on~$X$. 
Consider the adjoint functors
\begin{equation}
\label{eq:io-functors}
\io_*\colon \Db(X) \to \Db_X(\cX)
\qquad\text{and}\qquad 
\io^*\colon \Db_X(\cX) \to \Dp(X),
\end{equation}
(note that~$\Db_X(\cX) \subset \Db(\cX) = \Dp(\cX)$ since~$\cX$ is smooth, hence the image of~$\io^*$ is contained in~$\Dp(X)$).

\begin{lemma}
\label{lem:dense-images}
If~\eqref{eq:normal-bundle-triv} holds, the functors~\eqref{eq:io-functors} have dense images.
\end{lemma}
\begin{proof}
To show that~$\io_*$ has dense image it suffices to note that every object~$F \in \Db_X(\cX)$
has cohomology sheaves set-theoretically supported on~$X$, 
and these sheaves admit a filtration by sheaves which are scheme-theoretically supported on~$X$ 
(thus we see in this case that the image of~$\io_*$ already generates~$\Db_X(\cX)$ as triangulated category
and adding direct summands is not needed).

To show that~$\io^*$ has dense image it suffices to check that the composition~$\io^* \circ \io_*$ has dense image. 
We first show that for every~$G \in \Db(\cX)$ the object~$\io^*(G)$ lies in~$\thick(\io^*(\io_*(\Db(X))))$.
Indeed, the projection formula and the standard resolution 
\begin{equation}
\label{eq:koszul}
0 \to \cO_\cX(-X) \to \cO_\cX \to \io_*\cO_X \to 0
\end{equation}
imply that
\begin{equation*}
\io^* \io_* \io^*(G) \cong 
\io^*(G \otimes \io_*\cO_X) \cong
\io^*\Cone(G \otimes \cO_\cX(-X) \to G) \cong
\Cone(\io^*G \xrightarrow{\ 0\ } \io^*G) \cong
\io^*(G) \oplus \io^*(G)[1].
\end{equation*}
Here the map in the first cone is given by the equation of~$X \subset \cX$,
hence its restriction to~$X$ (the map in the second cone) is zero,
and we have used~\eqref{eq:normal-bundle-triv}.
This proves that~$\io^*(G)$ is a direct summand of~$\io^* \io_* \io^*(G)$,
hence belongs to the category~$\thick(\io^*(\io_*(\Db(X))))$.

Applying the result above to~$G = \cL^{\otimes k}$, with~$\cL$ an ample line bundle on~$\cX$ and~$k \in \Z$, 
we see that~$\thick(\io^*(\io_*(\Db(X))))$ contains a classical generator of~$\Db(X)$, see~\cite[Theorem~4]{Or09}),
so that~$\io^* \circ \io_*$ has dense image.
\end{proof}

Now we are ready to prove the theorem.

\begin{proof}[Proof of Theorem~\textup{\ref{thm:bijection-subcat-deform}}]
We prove the theorem in a sequence of steps.

\smallskip

{\bf Step 1. Construction of the diagram~\eqref{eq:upsilon-diagram}.}
The top arrow in the diagram is the map~\eqref{eq:ups-pf-adm} defined in Lemma~\ref{lem:ups-pf-Bijection}.
To define the other two arrows we use Lemma~\ref{lem:Gamma-adjoints};
indeed, the functors~$\io_*$ and~$\io^*$ have dense images by Lemma~\ref{lem:dense-images},
and since~$X \subset \cX$ is a Cartier divisor linearly equivalent to zero
we have distinguished triangles
\begin{equation}
\label{eq:rrr-triangle}
\io^*\io_*F \to F \to F[2]
\qquad\text{and}\qquad 
G \to \io_*\io^* G \to G[1]
\end{equation}
for any~$F \in \Db(X)$ and~$G \in \Db_X(\cX)$.
Therefore, the conditions of Lemma~\ref{lem:Gamma-adjoints} are satisfied for the functors~$\io_*$ and~$\io^*$,
hence the functors~$\Upsilon_{\io_*}$ and~$\Upsilon_{\io^*}$ preserve semiorthogonal decompositions,
and hence induce maps of the sets of admissible subcategories by Lemma~\ref{lem:sod-adm}.

\smallskip

{\bf Step 2. Commutativity of the diagram~\eqref{eq:upsilon-diagram}.}
Let $\cA \in \Adm(\Db(X))$
and consider the semiorthogonal decomposition~$\Db(X) = \langle \cA, \cB \rangle$.
Applying the maps in the diagram we obtain
two semiorthogonal decompositions
\begin{equation}\label{eq:two-decomp}
\Dperf(X) = \langle \cA \cap \Dperf(X), \cB \cap \Dperf(X) \rangle 
\qquad\text{and}\qquad 
\Dperf(X) = \langle
\Upsilon_{\io^*}(\Upsilon_{\io_*}(\cA)),
\Upsilon_{\io^*}(\Upsilon_{\io_*}(\cB))
\rangle.
\end{equation}
We will check that they coincide.

On the one hand, the first 
triangle in~\eqref{eq:rrr-triangle} 
implies that~$\io^*(\io_*(\cA)) \subset \cA$ and~$\io^*(\io_*(\cB)) \subset \cB$,
and since~$\cA$ and~$\cB$ are
idempotent complete, it follows that~$\Upsilon_{\io^*}(\Upsilon_{\io_*}(\cA)) \subset \cA$
and~$\Upsilon_{\io^*}(\Upsilon_{\io_*}(\cB)) \subset \cB$.
On the other hand, $\Upsilon_{\io^*}(\Upsilon_{\io_*}(\cA))$, 
$\Upsilon_{\io^*}(\Upsilon_{\io_*}(\cB)) \subset \Dp(X)$
by definition of~$\Upsilon_{\io^*}$ and~$\Upsilon_{\io_*}$.
Combining these inclusions, we deduce that
\begin{equation*}
\Upsilon_{\io^*}(\Upsilon_{\io_*}(\cA)) \subset \cA \cap \Dperf(X) = \Upsilon^\pf(\cA),
\qquad 
\Upsilon_{\io^*}(\Upsilon_{\io_*}(\cB)) \subset \cB \cap \Dperf(X) = \Upsilon^\pf(\cB).
\end{equation*}
Thus, the components of the second decomposition in~\eqref{eq:two-decomp}
are contained in the components of the first.
Therefore, the decompositions coincide and in particular 
we see that
\begin{equation*}
\Upsilon_{\io^*}(\Upsilon_{\io_*}(\cA)) = \cA \cap \Dperf(X),
\end{equation*}
hence the diagram commutes.

\smallskip

{\bf Step 3. Partial injectivity of~$\Upsilon_{\io^*}$.}
Assume we have an inclusion~$\cA_1 \subset \cA_2$ of admissible subcategories in~$\Db_X(\cX)$.
In this step we check that~$\Upsilon_{\io^*}(\cA_1) = \Upsilon_{\io^*}(\cA_2)$ implies~$\cA_1 = \cA_2$.

Since~$\cA_1$ is admissible 
in~$\Db_X(\cX)$ it is also admissible in~$\cA_2$
and it suffices to show that~${}^\perp\cA_1 \cap \cA_2 = 0$. 
Take any~$F \in {}^\perp\cA_1 \cap \cA_2$.
Since~$\Upsilon_{\io^*}$ preserves semiorthogonal decompositions, 
the inclusion~$F \in {}^\perp\cA_1$ implies~$\io^*F \in {}^\perp(\Upsilon_{\io^*}(\cA_1))$.
On the other hand, the inclusion~$F \in \cA_2$ implies~$\io^*F \in \Upsilon_{\io^*}(\cA_2) = \Upsilon_{\io^*}(\cA_1)$.
Combining these, we conclude~$\io^*F = 0$, and since~$F$ is set-theoretically supported on~$X$, in fact~$F = 0$.

\smallskip

{\bf Step 4. Surjectivity of~$\Upsilon_{\io_*}$.}
Take any $\cA \in \Adm(\Db_X(\cX))$
and consider the semiorthogonal decomposition~$\Db_X(\cX) = \langle \cA, \cB \rangle$.
Applying the map~$\Upsilon_{\io^*}$ we obtain
\begin{equation}
\label{eq:dpx-first}
\Dp(X) = \langle \Upsilon_{\io^*}(\cA), \Upsilon_{\io^*}(\cB) \rangle,
\end{equation}
with admissible subcategory~$\Upsilon_{\io^*}(\cA) \subset \Dp(X)$.
Since~$\Upsilon^\pf$ is bijective by Lemma~\ref{lem:ups-pf-Bijection}
there is an admissible subcategory~$\tcA \subset \Db(X)$ such that
\begin{equation}
\label{eq:ups-pf-tca}
\Upsilon^\pf(\tcA) = 
\Upsilon_{\io^*}(\cA).
\end{equation}
We claim that~$\Upsilon_{\io_*}(\tcA) \subset \cA$.
Indeed, extending~$\tcA$ to a semiorthogonal decomposition~$\Db(X) = \langle \tcA, \tcB \rangle$ 
and applying the map~$\Upsilon^\pf$ we obtain (using Lemma~\ref{lem:ups-pf-Bijection}) a semiorthogonal decomposition
\begin{equation}
\label{eq:dpx-second}
\Dp(X) = \langle \Upsilon^\pf(\tcA), \Upsilon^\pf(\tcB) \rangle,
\end{equation}
and since its first component coincides with the first component in~\eqref{eq:dpx-first}, 
the second components coincide as well, i.e., $\tcB \cap \Dp(X) = \Upsilon^\pf(\tcB) = \Upsilon_{\io^*}(\cB)$.
In particular, $\io^*(\cB) \subset \tcB$, hence~$\Hom(\io^*(\cB), \tcA) = 0$, 
and therefore by adjunction~$\Hom(\cB, \io_*(\tcA)) = 0$,
which implies the required inclusion~$\Upsilon_{\io_*}(\tcA) \subset \cB^\perp = \cA$.

We finally note that both~$\Upsilon_{\io_*}(\tcA)$ and~$\cA$ are admissible subcategories in~$\Db_X(\cX)$ 
(the first because~$\tcA$ is admissible and~$\Upsilon_{\io_*}$ preserves semiorthogonal decompositions,
the second by assumption)
and
\begin{equation*}
\Upsilon_{\io^*}(\cA) = \Upsilon^\pf(\tcA) = \Upsilon_{\io^*}(\Upsilon_{\io_*}(\tcA)),
\end{equation*}
where the first is~\eqref{eq:ups-pf-tca} and the second equality holds by Step~2.
Since~$\Upsilon_{\io_*}(\tcA) \subset \cA$, partial injectivity of~$\Upsilon_{\io^*}$ proved in Step~3 
implies~$\Upsilon_{\io_*}(\tcA) = \cA$,
which proves the surjectivity of the map~$\Upsilon_{\io_*}$.

\smallskip

{\bf Conclusion.}
First, $\Upsilon^\pf$ is a bijection by Lemma~\ref{lem:ups-pf-Bijection}.
The commutativity of the diagram~\eqref{eq:upsilon-diagram} (proved in Step~2) 
implies that~$\Upsilon_{\io_*}$ is injective. 
As we checked in Step~4 that it is surjective, we conclude that~$\Upsilon_{\io_*}$ is bijective.
Finally, commutativity of the diagram and bijectivity of~$\Upsilon^\pf$ and~$\Upsilon_{\io_*}$
imply bijectivity of~$\Upsilon_{\io^*}$.
The fact that~$\Upsilon_{\io^*}$ and~$\Upsilon_{\io_*}$
preserve semiorthogonal decompositions has been explained in Step~1,
and the fact that~$\Upsilon^\pf$ preserves semiorthogonal decompositions with an admissible component
is explained in Lemma~\ref{lem:ups-pf-Bijection}.
\end{proof}

We have the following particular case of Theorem~\ref{thm:bijection-subcat-deform},
relevant for applications in~\cite{KSabs}.

\begin{corollary}
\label{cor:smoothing}
If~$f \colon \cX \to B$ is a flat projective morphism from a smooth quasiprojective variety over a perfect field 
to a smooth curve, $o \in B$ is a point, and~$X = \cX_o$ is the central fiber, 
there is a commutative diagram of bijective maps~\eqref{eq:upsilon-diagram}
preserving semiorthogonal decompositions with an admissible component.
\end{corollary}

\begin{proof}
It is enough to note that shrinking~$B$ we may assume that the point~$o$ is linearly equivalent to zero, 
hence the fiber~$X = \cX_o$ is also linearly equivalent to zero, and then apply Theorem~\ref{thm:bijection-subcat-deform}.
\end{proof}

\begin{remark}
One can interpret Theorem~\ref{thm:bijection-subcat-deform} in terms of deformation theory of admissible
subcategories in~$\Dp(X)$.
Namely, Theorem~\ref{thm:bijection-subcat-deform} says that
every admissible subcategory in~$\Dp(X)$
has a unique extension to the formal neighbourhood of $X \subset \cX$.
This partially generalizes \cite[Theorem 7.1]{BOR} to  families
where the central fiber is not smooth.
\end{remark}

To finish the paper we explain the relation of Theorem~\ref{thm:bijection-subcat-deform}
(or rather Corollary~\ref{cor:smoothing}) 
to the deformation absorption property introduced in~\cite{KSabs}.
Recall that~$\rP \in \Db(X)$ is a $\Pinfty{2}$-object if~$\Ext^*(\rP, \rP) \simeq \bk[\uptheta]$
with~$\deg(\uptheta) = 2$, see~\cite[Definition~2.6, Remark~2.7]{KSabs},
and the simplest example of such an object is the simple module~$\kk_\sA$ in~$\Db(\sA_1)$,
where~$\sA_1 = \kk[\eps]/(\eps^2)$, $\deg(\eps) = -1$, is the dg-algebra from Example~\ref{ex:Ap}.
In fact, the subcategory in~$\Db(X)$ generated by any $\Pinfty{2}$-object is equivalent to~$\Db(\sA_1)$,
see~\cite[Lemma~2.10]{KSabs}.
In~\cite[Theorem~1.8]{KSabs} we proved that any $\Pinfty{2}$-object on the special fiber of a smoothing
gives an exceptional object on the total space.
The next result shows that this correspondence is a bijection.
We formulate it in terms of subcategories generated by these objects.

\begin{corollary}
\label{cor:bijection-pinfty}
Assume the situation of Theorem~\textup{\ref{thm:bijection-subcat-deform}}.
Then the maps~$\Upsilon_{\io_*}$, $\Upsilon_{\io^*}$ and~$\Upsilon^\pf$ define bijections between the following sets:
\begin{aenumerate}
\item 
\label{it:cp}
admissible subcategories~$\cP \subset \Db(X)$ such that~$\cP \simeq \Db(\sA_1)$;
\item 
\label{it:ce}
admissible subcategories~$\cE \subset \Db_X(\cX)$ such that~$\cE \simeq \Db(\kk)$; 
\item 
\label{it:cm}
admissible subcategories~$\cM \subset \Dp(X)$ such that~$\cM \simeq \Dp(\sA_1)$.
\end{aenumerate}
\end{corollary}

\begin{proof}
Let~$\Sub_\cP(\Db(X))$, $\Sub_\cE(\Db_X(\cX))$, and~$\Sub_\cM(\Dp(X))$ 
be the sets of subcategories described in~\ref{it:cp}, \ref{it:ce}, and~\ref{it:cm}, respectively.
First, we check that
\begin{equation*}
\Upsilon_{\io_*}(\Sub_\cP(\Db(X)) \subset \Sub_\cE(\Db_X(\cX))
\qquad\text{and}\qquad 
\Upsilon_{\io^*}(\Sub_\cE(\Db_X(\cX)) \subset \Sub_\cM(\Dp(X)).
\end{equation*}

If~$\cP \subset \Db(X)$, $\cP \simeq \Db(\sA_1)$, 
the object~$\rP \in \cP$ corresponding to the simple~$\sA_1$-module~$\kk_\sA$ is a $\Pinfty2$-object,
see~\cite[Lemma~2.4, Definition~2.6, and Remark~2.7]{KSabs}.
Then~$\rE \coloneqq \io_*\rP$ is an exceptional object supported on~$X$ by~\cite[Theorem~1.8]{KSabs}; 
therefore the admissible subcategory
\begin{equation*}
\Upsilon_{\io_*}(\cP) = 
\thick(\io_*\cP) = \thick(\io_*\rP) = \thick(\rE) = \langle \rE \rangle
\end{equation*}
is equivalent to~$\Db(\kk)$.

If~$\cE \subset \Db_X(\cX)$, $\cE \simeq \Db(\kk)$, it is generated by an exceptional object~$\rE \in \Db_X(\cX)$.
Using~\eqref{eq:koszul} we obtain
\begin{equation*}
\io_*\io^*\rE \cong
\rE \otimes \io_*\cO_X \cong
\Cone(\rE(-X) \to \rE) \cong
\Cone(\rE \to \rE),
\end{equation*}
where the map in the right side is induced by the equation of~$X \subset \cX$.
Therefore
\begin{equation*}
\Ext^\bullet(\io^*\rE, \io^*\rE) \cong
\Ext^\bullet(\rE, \io_*\io^*\rE) \cong
\Ext^\bullet(\rE, \Cone(\rE \to \rE)) \cong 
\Cone(\kk \to \kk).
\end{equation*}
If the morphism in the right-hand side is non-trivial, we obtain~$\Ext^\bullet(\io^*\rE, \io^*\rE) = 0$, 
hence~$\io^*E = 0$, which is absurd because~$\rE$ is set-theoretically supported on~$X$.
Therefore, the morphism is zero, hence~$\Ext^\bullet(\io^*\rE, \io^*\rE) \cong \kk \oplus \kk[1] \cong \sA_1$.
Since the dg-algebra~$\sA_1$ is intrinsically formal (see, e.g., \cite[Lemma~2.1]{KSabs}),
it follows that
\begin{equation*}
\Upsilon_{\io^*}(\cE) = 
\thick(\io^*\cE) = \thick(\io^*\rE) \simeq 
\Dp(\sA_1).
\end{equation*}

Thus, $\Upsilon_{\io_*}$ and~$\Upsilon_{\io^*}$ induce maps between our sets,
and these maps are injective by Theorem~\ref{thm:bijection-subcat-deform}.
Moreover, commutativity of the diagram~\eqref{eq:upsilon-diagram} implies 
that their composition is the map
\begin{equation}
\label{eq:upf}
\Upsilon^\pf \colon \Sub_\cP(\Db(X)) \to \Sub_\cM(\Dp(X)),
\end{equation}
which is also injective. 
Applying Corollary~\ref{cor:bij-X-A} we see that it is bijective.

Furthermore, surjectivity of~\eqref{eq:upf} and commutativity of the diagram~\eqref{eq:upsilon-diagram}
imply surjectivity of the map~$\Upsilon_{\io^*} \colon \Sub_\cE(\Db_X(\cX)) \to \Sub_\cM(\Dp(X))$,
and since this map is also injective, it is bijective.
Finally, commutativity of~\eqref{eq:upsilon-diagram} 
implies that~$\Upsilon_{\io_*} \colon \Sub_\cP(\Db(X)) \to \Sub_\cE(\Db_X(\cX))$ is also bijective.
\end{proof}

The following consequence of the above bijection is particularly interesting.

\begin{corollary}
In the situation of Corollary~\textup{\ref{cor:bijection-pinfty}} 
any exceptional object in the category~$\Db_X(\cX)$ 
is scheme-theoretically supported on~$X$.
\end{corollary}

\begin{proof}
If~$\rE$ is an exceptional object in~$\Db_X(\cX)$,
Corollary~\textup{\ref{cor:bijection-pinfty}} proves that~\mbox{$\rE \cong \io_*\rP$}, 
where~$\rP$ is a $\Pinfty2$-object in~$\Db(X)$;
in particular, $\rE$ is scheme-theoretically supported on~$X$.
\end{proof}


\end{document}